\documentclass[11pt,a4paper]{article}
\usepackage{ifpdf}
\usepackage[utf8]{inputenc}
\usepackage[T1]{fontenc}
\usepackage{graphicx}
\usepackage{color}
\usepackage{textcomp}
\usepackage{amsmath}
\usepackage{amssymb}
\usepackage{amsthm}
\usepackage{geometry}
\usepackage[all]{xy}

\usepackage[usenames,dvipsnames]{pstricks}
\usepackage{epsfig}
\usepackage{pst-grad} 
\usepackage{pst-plot} 
\title{Constraint equations for $3+1$ vacuum Einstein equations with a translational space-like Killing field in the asymptotically flat case}
\author{Cécile Huneau}

\ifpdf
\fi

\newtheorem{thm}{Theorem}[section]
\newtheorem{prp}[thm]{Proposition}
\newtheorem{cor}[thm]{Corollary}
\newtheorem{lm}[thm]{Lemma}
\newtheorem{df}[thm]{Definition}
\newtheorem{rk}[thm]{Remark}

\newcommand{\m}[1]{\mathbb{#1}}
\newcommand{\q}[1]{\mathcal{#1}}

\newcommand{\wht}[1]{\widetilde{#1}}
\newcommand{\gra}[1]{\mathbf{#1}}
\newcommand{\grat}[1]{\mathbf{\widetilde{#1}}}

\begin{document}

\maketitle
\thanks{Ecole Normale Supérieure, 45 rue d'Ulm, 75005 Paris, email : cecile.huneau@ens.fr}
\begin{abstract}
 We solve the Einstein constraint equations for a $3+1$ dimensional vacuum space-time with a space-like translational Killing field. 
The presence of a space-like translational Killing field allows for a reduction of the $3+1$ dimensional problem to a $2+1$ dimensional one.
Vacuum Einstein equations with a space-like translational Killing field  have been studied by Choquet-Bruhat and Moncrief in the compact case. In the case 
where an additional rotational symmetry is added, the problem has a long history (see \cite{beck}, \cite{asht}, \cite{chru}). In this paper we consider the asymptotically flat case.
This corresponds to solving a nonlinear elliptic system on $\m R^2$. The main difficulty in that case is due to the delicate inversion of the Laplacian on
$\m R^2$. In particular, we have to work in the non-constant mean curvature setting, which enforces us to consider the intricate coupling of the Einstein constraint equations.

\end{abstract}

\section{Introduction}
Einstein equations can be formulated as a Cauchy problem whose initial data must satisfy compatibility conditions known as the constraint equations.
In this paper, we will consider the constraint equations for the vacuum Einstein equations, in the particular case where the space-time possesses a 
space-like translational Killing field. It allows for a reduction of the $3+1$ dimensional
problem to a $2+1$ dimensional one. This symmetry has been studied by Choquet-Bruhat and Moncrief in \cite{choquet} (see also \cite{livrecb})
in the case of a
space-time of the form $\Sigma \times \mathbb{S}^1  \times \mathbb{R}$, where $\Sigma$ is a compact two dimensional manifold of genus $G\geq 2$, and $\mathbb{R}$
is the time axis, with a space-time metric independent of the $\mathbb{S}^1$ coordinate. 
They prove the existence of global solutions corresponding to perturbation of particular expanding initial data.

In this paper we consider a space-time of the form $\mathbb{R}^2 \times \mathbb{R} \times \mathbb{R}$, symmetric with respect to the third coordinate.
Minkowski space-time is a particular solution of vacuum Einstein equations which exhibits this symmetry. Since the celebrated work of Christodoulou and Klainerman (see \cite{ck}), we know that Minkowski space-time is stable, that is to say asymptotically flat perturbations of the trivial initial data lead to global solutions converging to Minkowski space-time.
It is an interesting problem to ask whether
the stability also holds in the setting of perturbations of Minkowski space-time with a space-like translational Killing field. Let's note that it is not included in the work of Christodoulou and Klainerman. However, it is crucial, before considering this problem, 
to ensure the existence of compatible initial data, i.e. the existence of solutions to the constraint equations. This is the purpose of the present paper.

In the compact case, if one looks for solutions with constant mean curvature, as it is done in \cite{choquet}, the issue of solving the constraint equations is straightforward. 
Every metric on a compact manifold of genus $G\geq 2$ is conformal
to a metric of scalar curvature $-1$. As a consequence, it is possible to decouple the system into elliptic
scalar equations of the form $\Delta u = f(x,u)$ with $\partial_u f >0$, for which existence results  are standard (see for example chapter $14$ in \cite{tay}).

The asymptotically flat case is more challenging. First, the definition of an asymptotically flat manifold is not so clear in two dimension. 
In \cite{beck}, \cite{asht}, \cite{chru} radial solutions of the $2+1$ dimensional problem with an angle at space-like infinity are constructed.
In particular, these solutions do not tend to the Euclidean metric at space-like infinity. Moreover, the behaviour of the Laplace operator on $\m R^2$ makes the issue of finding solutions to the constraint equations
more intricate.

\subsection{Reduction of the Einstein equations}
Before discussing the constraint equations, we first briefly recall the form of the Einstein equations in the presence of a space-like translational 
Killing field. We follow here the exposition
in \cite{livrecb}.
A metric $^{(4)}\mathbf{g}$ on  $\mathbb{R}^2 \times \m R  \times \mathbb{R}$ admitting $\partial_3$ as a Killing field can be written
$$^{(4)}\mathbf{g} = \grat g + e^{2\gamma}(dx^3 +A_\alpha dx^\alpha)^2, $$
where $\grat g$ is a Lorentzian metric on $\mathbb{R}^{1+2}$, $\gamma$ is a scalar function on $\mathbb{R}^{1+2}$, $A$ is a $1$-form on
$\mathbb{R}^{1+2}$ and $x^\alpha$, $\alpha=0,1,2$, are the coordinates on $\mathbb{R}^{1+2}$. Since $\partial_3$ is a Killing field, $\mathbf{g}$, $\gamma$ and $A$ do not depend on $x^3$.
We set $F=dA$, where $d$ is the exterior differential. $F$ is then a $2$-form. Let also $^{(4)}\mathbf{R}_{\mu \nu}$ denote the Ricci tensor associated to $^{(4)}\mathbf{g}$.
$\grat R_{\alpha \beta}$ and $\grat D$ are respectively the Ricci tensor and the covariant derivative associated to $\grat g$.

With this metric, the vacuum Einstein equations
$$^{(4)}\mathbf{R}_{\mu \nu} = 0, \; \mu, \nu = 0,1,2,3$$
can be written in the basis $(dx^\alpha, dx^3+A_\alpha dx^\alpha)$  (see \cite{livrecb} appendix VII)
\begin{align}
\label{rab}0= ^{(4)}\gra R_{\alpha \beta} &= \grat R_{\alpha \beta}-\frac{1}{2}e^{2\gamma}{F_\alpha}^\lambda F_{\beta \lambda}
 -\grat D_\alpha \partial_\beta \gamma -\partial_\alpha \gamma \partial_\beta \gamma,\\
\label{ra3} 0=^{(4)}\gra R_{\alpha 3} &= \frac{1}{2}e^{-\gamma} \grat D_\beta (e^{3\gamma}{F_\alpha}^\beta),\\
\label{r33} 0=^{(4)}\gra R_{33}&= -e^{-2\gamma} \left( -\frac{1}{4}e^{2\gamma}F_{\alpha \beta}F^{\alpha \beta} + \grat g^{\alpha \beta}\partial_\alpha \gamma \partial_\beta \gamma
 + \grat g^{\alpha \beta}\grat D_\alpha \partial_\beta \gamma\right).
\end{align} 

The equation \eqref{ra3} is equivalent to 
$$d(\ast e^{3\gamma}F)=0$$
where $\ast e^{3\gamma}F$ is the adjoint one form associated to $e^{3\gamma}F$.
This is equivalent, on $\m R^{1+2}$, to the existence of a potential $ \omega$ such that
$$\ast e^{3\gamma}F=d\omega.$$
Since $F$ is a closed $2$-form, we have $dF=0$. By doing the conformal change of metric $\grat g = e^{-2\gamma}\gra g$, 
this equation, together with the equations \eqref{rab} and \eqref{r33}, yield the following system,

\begin{align}
\label{wavebis} &\Box_{\gra g} \omega -4 \partial^\alpha \gamma \partial_\alpha \omega=0,\\
 \label{wave}& \Box_\mathbf{g} \gamma +\frac{1}{2}e^{-4\gamma} \partial^\alpha \omega \partial_\alpha \omega= 0, \\
\label{einst2} &\mathbf{R}_{\alpha \beta} = 2\partial_\alpha \gamma\partial_\beta \gamma +\frac{1}{2}e^{-4\gamma}\partial_\alpha \omega \partial_\beta \omega, \; \alpha,\beta = 0,1,2,
\end{align}
where $\Box_\mathbf{g}$ is the d'Alembertian\footnote{$\Box_\mathbf{g}$ is the
Lorentzian equivalent of the Laplace-Beltrami operator in Riemannian geometry. In a coordinate system,
we have $\Box_\mathbf{g} u = \frac{1}{\sqrt{|\mathbf{g}|}}\partial_\alpha(\mathbf{g}^{\alpha \beta}\sqrt{|\mathbf{g}|}\partial_\beta u)$.} in the metric $\mathbf{g}$
 and  $\mathbf{R}_{\alpha \beta}$ is the Ricci tensor associated to $\mathbf{g}$.
We introduce the following notation
\begin{equation}\label{defubis}
 u \equiv (\gamma,\omega),
\end{equation}
together with the scalar product
\begin{equation}\label{defu}
\partial_\alpha u. \partial_\beta u=2\partial_\alpha \gamma \partial_\beta \gamma + \frac{1}{2}e^{-4\gamma}\partial_\alpha \omega \partial_\beta \omega.
\end{equation}

We consider the Cauchy problem for the equations \eqref{wavebis}, \eqref{wave} and \eqref{einst2}.
As it is in the case for the $3+1$ Einstein equation, the initial data for \eqref{wavebis}, \eqref{wave} and \eqref{einst2} cannot be prescribed arbitrarily. They have to satisfy constraint equations.

\subsection{Constraint equations}
 We can write the metric $\mathbf{g}$ under the form
\begin{equation}\label{metrique}
\mathbf{g} = -N^2(dt)^2 +g_{ij}(dx^i +\beta^i dt)(dx^j + \beta^j dt),
\end{equation}
where the scalar function $N$ is called the lapse, the vector field $\beta$ is called the shift and $g$ is a Riemannian metric on $\m R^2$. 

We consider the initial space-like surface $\m R ^2 = \{t=0\}$. Let $T$ be the unit normal to $\m R ^2 = \{t=0\}$. We set
$$e_0=NT=\partial_t-\beta^j \partial_j.$$
We will use the notation
$$\partial_0=\mathcal{L}_{e_0}=\partial_t - \mathcal{L}_{\beta},$$
where $\mathcal{L}$ is the Lie derivative. With this notation, we have the following 
expression for the second fundamental form of $\m R^2$
$$K_{ij}=-\frac{1}{2N}\partial_0 g_{ij}.$$
We will use the notation 
$$\tau=g^{ij}K_{ij}$$
for the mean curvature. We also introduce the Einstein tensor
$$\mathbf{G}_{\alpha \beta}= \mathbf{R}_{\alpha \beta} - \frac{1}{2}\mathbf{R}\mathbf{g}_{\alpha \beta},$$
where $\mathbf{R}$ is the scalar curvature $\mathbf{R} = \mathbf{g}^{\alpha \beta}\mathbf{R}_{\alpha \beta}$. 
The constraint equations are given by
\begin{align}
 \label{contrmom} \mathbf{G}_{0j} &\equiv N(\partial_j \tau - D^i K_{ij})=\partial_0 u. \partial_j u, \; j=1,2,\\
\label{contrham} \mathbf{G}_{00} & \equiv \frac{N^2}{2}(R-|K|^2+ \tau^2)= \partial_0 u. \partial_0 u - \frac{1}{2}\mathbf{g}_{00} \mathbf{g}^{\alpha \beta}\partial_\alpha u \partial_\beta u,
\end{align}
where $D$ and $R$ are respectively the covariant derivative and the scalar curvature associated to $g$ (see \cite{livrecb} chapter VI for a derivation of \eqref{contrmom} and \eqref{contrham}).
Equation \eqref{contrmom} is called the momentum constraint and \eqref{contrham} is called the Hamiltonian constraint. If we came back to the $3+1$ problem, there should be four constraint equations. However, since the fourth would be obtained by taking $\alpha=0$ in \eqref{ra3}, it is trivially satisfied if we set $\ast e^{3\gamma}F=d\omega$.

We will look for $g$ of the form $g= e^{2\lambda}\delta$ where $\delta$ is the Euclidean metric on $\mathbb{R}^2$. There is no loss of generality since, up to a diffeomorphism, all
metrics on $\mathbb{R}^2$ are conformal to the Euclidean metric. We introduce the traceless part of $K$, 
$$H_{ij} = K_{ij} - \frac{1}{2}\tau g_{ij},$$
and following \cite{choquet} we introduce the quantity
$$\dot{u} = \frac{e^{2\lambda}}{N}\partial_0 u.$$
Then the equations \eqref{contrmom} and \eqref{contrham} take the form
\begin{align}
 \label{mombis}&\partial_i H_{ij} = -\dot{u}.\partial_j u + \frac{1}{2} e^{2\lambda} \partial_j \tau,\\
\label{hambis}& \Delta \lambda + e^{-2\lambda}\left(\frac{1}{2}\dot{u}^2 +\frac{1}{2}|H|^2\right)-e^{2\lambda}\frac{\tau^2}{4} + \frac{1}{2}|\nabla u|^2 = 0,
\end{align}
where here and in the remaining of the paper, we use the convention for the Laplace operator
$$\Delta=\partial^2_1+\partial^2_2.$$
\\

The aim of this paper is to solve the coupled system of nonlinear elliptic equations \eqref{mombis} and \eqref{hambis} on $\m R^2$ in the small
data case, that is to say when $\dot{u}$ and $\nabla u$ are small. A similar system can be obtained when studying the constraint equations in three
dimensions by using the conformal method, introduced by Lichnerowicz \cite{lich} and Choquet-Bruhat and York \cite{york}.
In the constant mean curvature (CMC) case, that is to say when one sets $\tau=0$,
the constraint equations decouple and
the main difficulty that remains is the study
of the scalar equation \eqref{hambis}, also called the Lichnerowicz equation\footnote{The resolution of this equation is closely linked to the Yamabe problem}. 
The  CMC solutions have been studied in \cite{york} and \cite{isem} for the compact case, and in \cite{cantor} for the asymptotically flat case.
There have been also some results concerning the coupled constraint equations, i.e. without setting $\tau$ constant
The near CMC solutions in the asymptotically flat case have been studied in \cite{yorkcb}.
The compact case has been studied in
\cite{holst}, \cite{maxwell} and \cite{humbert}. See also \cite{bartnik} for a review of these results.

In our case, the difficulty will arise from particular issues concerning the inversion of second order elliptic operators on $\m R^2$.
In particular, without special assumptions on $u$, it is not possible to set $\tau=0$ in the case of $\m R^2$.
Indeed, equation \eqref{mombis} induces for $H$ the asymptotic $|H|^2 \sim \frac{1}{r^2}$ as $r$ tends to infinity. Now,
it is known (see \cite{ni}) that an equation of the form
$$\Delta u + Re^{2u}+f= 0,$$
with $R,f\leq 0$ and $R\lesssim -\frac{1}{r^2}$ when $r$ tends to infinity, admits no solution. Therefore, we will be forced to carefully adjust the asymptotic behaviour of $\tau$ as $r$ tends to infinity,
to compensate the term $|H|^2$ in equation \eqref{hambis}, and to ensure that we remain in the range of the elliptic operators which come into play. 

\begin{rk}\label{consing}
The solution of equation \eqref{hambis} that we construct in this paper satisfies 
\begin{equation}
\label{asym}\lambda = -\alpha \ln(r)+o(1),
\end{equation}
as $r\rightarrow \infty$, with $\alpha>0$.
At first sight, this could seem to contradict the asymptotic flatness we are looking for.
However, we mentioned in the beginning of the introduction that it is not so clear what to expect as a definition of asymptotic flatness in $2+1$
dimension. The solutions of the evolution problem \eqref{wavebis}, \eqref{wave} and \eqref{einst2} with an additional rotational symmetry
and $\omega \equiv 0$, 
known as Einstein-Rosen waves, have been studied in \cite{beck} and \cite{asht}. These solutions exhibit a conical singularity at space-like infinity, 
that is to say the perimeter of a circle of radius $r$ asymptotically grows like $2\pi c r$ with $c<1$,
instead of $2\pi r$ in the Euclidean metric. 

Using a change of variable,
we observe that the asymptotic behavior \eqref{asym}
is equivalent to the presence of an asymptotic angle at space-like infinity. 
Indeed, if we make the change of coordinate $r'=\frac{r^{1-\alpha}}{1-\alpha}$ for $r$ large enough, then the metric
$$g \sim r^{-2\alpha}(dr^2 + r^2 d\theta^2), \quad r\rightarrow \infty$$
takes the form
$$g' \sim dr'^2 + (1-\alpha)^2r'^2d\theta^2, \quad r' \rightarrow \infty$$
which corresponds to a conical singularity at space-like infinity, with an angle
given by 
$$2\pi (1-\alpha).$$
Note that, since the constraint equations \eqref{contrmom} and \eqref{contrham} are independent of the choice of
coordinates, the metric $g'$ and the second fundamental form $K'$, obtained by performing the change of variables  $r'=\frac{r^{1-\alpha}}{1-\alpha}$ for $r$ large enough,
are still solutions of the constraint equations.
\end{rk}

We will do the following rescaling to avoid the $e^{2\lambda}$
and $e^{-2\lambda}$ factors
$$\breve{u}=e^{-\lambda}\dot{u}, \quad \breve{H}=e^{-\lambda}H, \quad \breve{\tau} = e^{\lambda}\tau.$$
Then the equations \eqref{mombis} and \eqref{hambis} become
\begin{align*}
&\partial_i \breve{H}_{ij} +\breve{H}_{ij}\partial_i \lambda  = -\breve{u}.\partial_j u + \frac{1}{2} \partial_j \breve{\tau}-\frac{1}{2} \breve{\tau}\partial_j \lambda,\\
& \Delta \lambda + \frac{1}{2}\breve{u}^2 + \frac{1}{2}|\nabla u|^2+\frac{1}{2}|\breve{H}|^2-\frac{\breve{\tau}^2}{4} = 0.
\end{align*}
To lighten the notations, we will omit the $\;\breve{}\;$ in the rest of the paper.

\section{Main result}
We are interested in the system of constraint equations  on $\mathbb{R}^2$
\begin{equation}\label{sys}
\left\{
\begin{array}{l}
\partial_i H_{ij} +H_{ij}\partial_i \lambda  = -\dot{u}.\partial_j u + \frac{1}{2} \partial_j \tau-\frac{1}{2} \tau\partial_j \lambda,\\
\Delta \lambda + \frac{1}{2}\dot{u}^2 + \frac{1}{2}|\nabla u|^2+\frac{1}{2}|H|^2-\frac{\tau^2}{4} = 0. \\
\end{array} \right.
\end{equation}
We look for solutions $(H,\lambda)$ where $H$ is a $2$-tensor, symmetric and traceless, and $\lambda$ is a scalar function. The function $u$ is defined by 
\eqref{defubis} and $\tau$ is a scalar function which may be chosen arbitrarily. We recall however from the end of the previous section 
that $\tau$ must be chosen carefully. In particular, $\tau$ can not be identically $0$ otherwise there may not be solutions. In our work, some parts of $\tau$ will be imposed. These are the quantities
$$\lim_{r\rightarrow \infty} \int_0^{2\pi}\tau(r,\theta) \cos(\theta)rd\theta,  \quad\lim_{r\rightarrow \infty} \int_0^{2\pi}\tau(r,\theta) \sin(\theta)rd\theta,$$
where $(r,\theta)$ are polar coordinates.
We will work in the case where we choose data for $u$ and $\tau$ which are small.

Before stating the main theorem, we recall several properties of weighted Sobolev spaces. 

\subsection{Weighted Sobolev spaces}
In the rest of the paper, $\chi(r)$ denotes a smooth non negative function such that 
$$0 \leq \chi \leq 1, \quad  \chi(r) =0 \; \text{for}\; r \leq 1, \quad \chi(r)=1 \; \text{for} \; r \geq 2.$$
We will also note $f \lesssim h $ when there exists a universal constant $C$ such that $f\leq Ch$.

\begin{df} Let  $m\in \m N$ and $\delta \in \mathbb{R}$. The weighted Sobolev space $H^m_\delta(\mathbb{R}^n)$ is the completion of $C^\infty_0$ for the norm 
 $$\|u\|_{H^m_\delta}=\sum_{|\beta|\leq m}\|(1+|x|^2)^{\frac{\delta +|\beta|}{2}}D^\beta u\|_{L^2}.$$
The weighted Hölder space $C^m_{\delta}$ is the complete space of $m$-times continuously differentiable functions with norm 
$$\|u\|_{C^m_{\delta}}=\sum_{|\beta|\leq m}\|(1+|x|^2)^{\frac{\delta +|\beta|}{2}}D^\beta u\|_{L^\infty}.$$
Let $0<\alpha<1$. The Hölder space $C^{m+\alpha}_\delta$ is the the complete space of $m$-times continuously differentiable functions with norm 
$$\|u\|_{C^{m+\alpha}_{\delta}}=\|u\|_{C^m_\delta} + \sup_{x \neq y, \; |x-y|\leq 1} \frac{|\partial^m u(x)-\partial^m u(y)|(1+|x|^2)^\frac{\delta}{2}}{|x-y|^\alpha}.$$
\end{df}

The following lemma is an immediate consequence of the definition.

\begin{lm} Let $m \geq 1$ and $\delta \in \mathbb{R}$. Then $u \in H^m_\delta$ implies $\partial_j u \in H^{m-1}_{\delta+1}$ for $j=1,..,n$.
\end{lm}
We first recall the Sobolev embedding with weights (see for example \cite{livrecb}, Appendix I). In the rest of this section, we assume $n=2$.

\begin{prp}\label{holder} Let $s,m \in \m N$. We assume $s >1$. Let $\beta \leq \delta +1$ and $0<\alpha<min(1,s-1)$. Then, we have the continuous embedding
$$H^{s+m}_{\delta}\subset C^{m+\alpha}_{\beta}.$$
\end{prp}

We will also need a product rule.

\begin{prp}\label{produit}
 Let $s,s_1,s_2 \in \m N$. We assume $s\leq \min(s_1,s_2)$ and $s<s_1+s_2-1$. Let $\delta<\delta_1+\delta_2+1$. Then $\forall (u,v) \in H^{s_1}_{\delta_1}\times H^{s_2}_{\delta_2}$,
$$\|uv\|_{H^s_\delta} \lesssim \|u\|_{H^{s_1}_{\delta_1}} \|v\| _{H^{s_2}_{\delta_2}}.$$
\end{prp}

The following simple lemma will be useful as well.

\begin{lm}\label{produit2}
Let $\alpha \in \mathbb{R}$ and $g \in L^\infty_{loc}$ be such that
$$|g(x)| \lesssim (1+|x|^2)^\alpha.$$
Then the multiplication by $g$ maps $H^0_{\delta}$ to $H^0_{\delta -2\alpha}$.
\end{lm}

We have the following theorem due to McOwen (see \cite{laplacien})

\begin{thm}\label{laplacien}(Theorem 0 in \cite{laplacien})
 Let $m\in \mathbb{N}$ and $-1+m<\delta<m$. The Laplace operator $\Delta :H^2_\delta \rightarrow H^0_{\delta+2}$ is an injection with closed range 
$$\left \{f \in H^0_{\delta+2}\; | \;\int fv =0 \quad \forall v \in \cup_{i=0}^m \mathcal{H}_i \right \},$$
where $\mathcal{H}_i$ is the set of harmonic polynomials of degree $i$.
Moreover, $u$ obeys the estimate
$$\|u\|_{H^2_{\delta}} \leq C(\delta)\|\Delta u\|_{H^0_{\delta+2}},$$
where $C(\delta)$ is a constant such that $C(\delta) \rightarrow +\infty$ when $\delta \rightarrow m_{-}$ and  $\delta \rightarrow (-1+m)_{+}$.
\end{thm}

\begin{cor} \label{coro} Let $-1<\delta<0$ and $f\in H^0_{\delta+2}$. Then there exists a solution $u$ of 
$$\Delta u =f$$
 which can be written 
$$u=\frac{1}{2\pi}\left(\int f\right)\chi(r)\ln(r) +v,$$
where $v \in H^2_{\delta}$ is such that
$\|v\|_{H^2_\delta} \leq C(\delta)\|f\|_ {H^0_{ \delta+2}}$.
\end{cor}

\begin{proof}
 Let $\Theta$ be a smooth function supported in $B(0,1)$ such that $\int \Theta = 2\pi$. Let $u_0$ be defined by 
$$u_0 (x) = \frac{1}{2\pi}\int \ln(|x-y|)\Theta (y) dy.$$
Then $u_0$ is a solution of $\Delta u_0=\Theta$.
We can write, for $|x|\geq 2$
$$u_0-\ln(|x|)=\frac{1}{2\pi}\int_{|y|\leq 1} \Theta(y)(\ln(|x-y)-\ln(|x|))dy,$$
therefore we have, for $|x|\geq 2$
$$|u_0-\ln(|x|)|\leq \frac{C(\Theta)}{|x|}.$$
Moreover,
$$\partial^i u_0-\partial^i\ln(|x|)=\frac{1}{2\pi}\int_{|y|\leq 1} \Theta(y)\partial^i(\ln(|x-y)-\ln(|x|))dy$$
and therefore, for $|x|\geq 2$
$$|\partial^i u_0-\partial^i\ln(|x|)|\leq \frac{C(\Theta)}{|x|^i}.$$
Besides, for $|x|\leq 2$, since $\ln$ is locally integrable, we have
$$\partial^i u_0 =  \frac{1}{2\pi}\int \ln(|y|)\partial^i\Theta (x-y) dy,$$
and therefore, for $|x|\leq 2$
$$|\partial^i u_0|\leq C(\Theta).$$
Consequently, we can write
$$u_0(x) = \chi(|x|)\ln(|x|) +\widetilde{u_0},$$
with $\widetilde{u_0} \in H^2_{\delta}$.
Theorem \ref{laplacien} implies that there exists $\widetilde{v} \in H^2_{\delta}$ solution of
$$\Delta \widetilde{v} = f-\frac{1}{2\pi}\left( \int f \right) \Theta.$$
Therefore 
$$ u = \frac{1}{2\pi}\left( \int f \right)\chi(r)\ln(r) +\frac{1}{2\pi}\left( \int f \right)\widetilde{u_0} + \widetilde{v}$$
is a solution of $\Delta u = f$.
To obtain the estimate of the corollary, it suffices to note that, for $f\in H^0_{\delta+2}$ we have
$$\int |f| = \int |f|\frac{(1+r^2)^{\frac{\delta}{2} + 1}}{(1+r^2)^{\frac{\delta}{2} + 1}} \lesssim \frac{1}{\sqrt{1+\delta}} \|f\|_{H^0_{\delta+2}}.$$
\end{proof}

\begin{cor}\label{regu}
 Let $s,m \in \mathbb{N}$ and $-1+m < \delta < m$. The Laplace operator $\Delta :H^{2+s}_\delta \rightarrow H^s_{\delta+2}$ is an injection with closed range 
$$\left \{f \in H^s_{\delta+2}\; | \; \int fv =0 \quad \forall v \in \cup_{i=0}^m \mathcal{H}_i \right \}.$$
Moreover, $u$ obeys the estimate
$$\|u\|_{H^{s+2}_{\delta}} \leq C(s,\delta)\|\Delta u\|_{H^s_{\delta+2}}.$$
\end{cor}

\begin{proof}
 We will proceed by induction. Note that Theorem \ref{laplacien} corresponds to the case $s=0$. We assume that the statement of the corollary 
holds true for some $s\in \mathbb{N}$ and all $m\in \mathbb{N}$, and we will prove that it holds true for $s+1$.
Let $m\in \mathbb{N}$ and $-1+m < \delta <m$. Let $f\in H^{s+1}_{\delta+2}$, such that $f$ belongs to the set
$$\left \{f \in H^0_{\delta+2} \; | \; \int fv =0 \quad \forall v \in \cup_{i=0}^m \mathcal{H}_i \right \}.$$
Then Theorem \ref{laplacien} provides a unique $u \in H^2_{\delta}$ such that $\Delta u = f$. In particular for $i=1,2$ we have
$$\Delta \partial_i u = \partial_i f.$$
Since $f\in H^{s+1}_{\delta+2}$, we have that $\partial_i f \in H^s_{\delta+3}$. Moreover, for all $v$, harmonic polynomial of degree $j\leq m+1$,
we have
$$\int (\partial_i f)v = -\int f \partial_i v =0,$$
because $\partial_i v$ is an harmonic polynomial of degree $j-1 \leq m$. Therefore, by induction, we have $\partial_i u \in H^{s+2}_{\delta+1}$ and
\begin{align*}
\|u\|_{H^{s+1+2}_{\delta}} &\lesssim \|u\|_{H^2_\delta} + \|\partial_1 u\|_{H^{s+2}_{\delta+1}} +\|\partial_2 u\|_{H^{s+2}_{\delta+1}}\\
&\leq C(\delta)\|f\|_{H^0_{\delta+2}} + C(s,\delta+1)\left(\|\partial_1 f\|_{ H^s_{\delta+3}}+\|\partial_2 f\|_{ H^s_{\delta+3}}\right)\\
&\leq C(s+1,\delta)\|f\|_{H^{s+1}_{\delta+2}}.
\end{align*}
 \end{proof}

\subsection{Main result}

In the rest of the paper,  $\delta$ will be a fixed real number such that 
$$-1<\delta<0.$$
\begin{df}
Let $\delta'\in \m R$ and $s \in \m N$. We note $\q H^s_{\delta'}$ the set of symmetric traceless tensor whose components are in 
$H^s_{\delta'}$.
\end{df}
We introduce the following traceless tensor, which is the traceless part of $rd\theta^2$
$$H_\theta = -\frac{ \chi(r)}{2r}\left(\begin{array}{cc}
                                          \cos(2\theta) & \sin(2\theta) \\ 
                                           \sin(2 \theta) & -\cos(2\theta)
                                         \end{array} \right),
$$
where $(r,\theta)$ are polar coordinates.
The following theorem is our main result.

\begin{thm}\label{main}Let $\dot{u}^2, |\nabla u|^2 \in H^0_{\delta +2}$, $\widetilde{\tau} \in H^1_{\delta+1}$ and $b\in L^\infty(\m S^1)$ such that
$$\int_{\m S^1} b(\theta)\cos(\theta)d\theta=\int_{\m S^1} b(\theta)\sin(\theta)d\theta =0.$$
We note
$$\varepsilon = \int \dot{u}^2 +|\nabla u|^2.$$
We assume 
$$\|\dot{u}^2\| _{H^0_{\delta+2}} + \||\nabla u|^2\|_{H^0_{\delta+2}} +\|\widetilde{\tau}\|_{H^1_{\delta+1}}+\|b\|_{L^\infty}\lesssim \varepsilon.$$
If $\varepsilon>0$ is small enough, there exist $\alpha, \rho, \eta \in \mathbb{R}$, a scalar function $\widetilde{\lambda} \in H^2_\delta$ and 
a traceless symmetric tensor $\widetilde{H} \in \q H^1_{\delta +1}$ 
such that, if we note
\begin{align*}
H =& \left(b(\theta)+\rho\cos(\theta-\eta)\right)H_\theta +\widetilde{H},\\
\lambda =&-\alpha \chi(r) \ln(r) + \widetilde{\lambda},
\end{align*}
where $(r,\theta)$ are polar coordinates, then $\lambda, H$ is a solution of \eqref{sys} with
$$  \tau = \frac{\chi(r)}{r}\left(b(\theta)+\rho\cos(\theta-\eta)\right) +\widetilde{\tau}.$$
Moreover, $\alpha, \rho, \eta, \widetilde{\lambda}, \widetilde{H}$ are unique. Finally, $\alpha$, $\rho$, $\eta$ are such that
\begin{align*}
 \alpha&=  \frac{1}{4\pi} \int \left(\dot{u}^2 +|\nabla u|^2\right) +O(\varepsilon^2) ,\\
\rho\cos(\eta) &= \frac{1}{\pi}\int \dot{u}\partial_1 u + O(\varepsilon ^{2}) ,\\
\rho \sin(\eta) &= \frac{1}{\pi}\int \dot{u}\partial_2 u + O(\varepsilon ^{2}) ,\\
\end{align*}
and we have the estimates $\| \widetilde{\lambda}\|_{H^2_\delta} +|\alpha|\lesssim \varepsilon$ and $\| \widetilde{H} \|_{\q H^1_{\delta+1}}+|\rho| \lesssim \varepsilon$.
\end{thm}

The following corollary is an immediate consequence of Theorem \ref{main} and Corollary \ref{regu}.

\begin{cor}\label{regu2}
Let $s\in \mathbb{N}$ and assume $\dot{u}^2, |\nabla u|^2 \in H^s_{\delta+2}$, $b \in W^{s,\infty}(\m S^1)$
and $\widetilde{\tau} \in H^{s+1}_{\delta+1}$. Let $\varepsilon$ be defined as in Theorem \ref{main}.
Then the conclusion of Theorem \ref{main} holds and we have furthermore $\widetilde{\lambda} \in H^{s+2}_{\delta}$,
 $\widetilde{H} \in \q H^{s+1}_{\delta+1}$, with the estimates 
$$\|\widetilde{\lambda}\|_{H^{s+2}_\delta}+\|\widetilde{H}\|_{\q H^{s+1}_{\delta+1}} \lesssim 
\|\dot{u}^2\|_{H^s_{\delta+2}}+\||\nabla u|^2\|_{H^s_{\delta+2}} +\|\widetilde{\tau}\|_{H^{s+1}_{\delta+1}}+\|b\|_{W^{s,\infty}}.$$
\end{cor}

\paragraph{Comments on Theorem \ref{main}}
\begin{enumerate}
\item The trivial asymptotically flat solution to the Einstein vacuum equations with a space-like translational Killing field \eqref{wavebis}, \eqref{wave} and \eqref{einst2}
is obtained by taking for $\gra g$ the Minkowski metric on $\m R^{1+2}$, and by setting $\omega= \gamma=0$. The corresponding initial data set 
is given by 
$$(\dot{u}^2= 0, |\nabla u|^2=0 , \tau=0, H=0, \lambda=0).$$
Theorem \ref{main} corresponds to the existence of solutions to the constraint equations which are small perturbations of 
$(0,0,0,0,0)$. An interesting open problem is the question of the non linear stability of the ``Minkowski space-time with a space-like translational Killing field'' under these perturbations\footnote{This is the analogue in dimension $2+1$ of the nonlinear stability of the Minkowski space-time
in dimension $3+1$, which has been established in the celebrated work of Christodoulou and Klainerman \cite{ck}.}.  
\item We solve here the constraint equations for small data. It is an interesting open problem to investigate the large data case.
\item  The logarithmic divergence in $\lambda$ does not contradict asymptotic flatness (see Remark \ref{consing}).
\item To understand where the special asymptotic structure of our solutions comes from, we can consider the space-time metric, given in $(t,r,\theta)$ coordinates by
$$\gra g_h = -dt^2 + r^{-2\alpha}(dr^2+(r-2h(\theta)r^{\alpha}t)^2d\theta^2).$$
This is a flat Lorentzian metric wherever it is well defined. The induced Riemannian metric on the surface $t=0$ is $r^{-2\alpha}\delta$, and we can calculate the second fundamental form, which is given
in $(x_1,x_2)$ coordinates by
$$K= \frac{h(\theta)r^{-\alpha}}{r}\left(\begin{array}{ll}
                                               \sin^2(\theta) & -\cos(\theta)\sin(\theta) \\
                                               -\cos(\theta)\sin(\theta) & \cos^2(\theta) 
                                              \end{array}\right).$$
Therefore, for such a metric, we have $\tau =\frac{r^\alpha}{r}h(\theta)$ and $H=h(\theta)r^{-\alpha}H_\theta$. By choosing $h(\theta)=b(\theta)+\rho\cos(\theta-\eta)$, we obtain exactly the asymptotic behavior predicted by Theorem \ref{main}. Since $\gra g_h$ is a flat space-time metric, the induced metric and second fundamental form on $\m R^2$ are special solutions of the vacuum constraint equations.
\item The function $b(\theta)$ is a free parameter. However, $\rho$ and $\eta$ must be chosen according to the asymptotic behavior $H=O(\frac{1}{r})$, induced by the fundamental solution of the Laplace operator in the first equation of \eqref{sys}.
The particular structure of the coupling between the momentum and the Hamiltonian constraints, explained in the previous comment, makes this cancellation possible, and allows us to perform 
a fixed point theorem
in weighted Sobolev spaces. 
\item The quantities $\alpha$, $\rho$ and $\eta$ are conserved by the flow of the Einstein equations. 
To see this, we note that $\rho,\eta$ can be expressed as
\begin{align*}
\rho\cos(\eta) & =\frac{1}{\pi} \lim_{r \rightarrow \infty} \int_0^{2\pi} \tau\cos(\theta) rd\theta,\\
\rho \sin(\eta) & =\frac{1}{\pi} \lim_{r \rightarrow \infty} \int_0^{2\pi} \tau\sin(\theta) rd\theta.
\end{align*}
The $(0,0)$ component of equation \eqref{einst2} can be written under the form
$$\partial_0 \tau = -e^{-2\lambda}\Delta N + \left(e^{-4\lambda}(|H|^2 + \dot{u}^2) + \frac{\tau^2}{2}\right)N.$$
This yields $\partial_t \tau = O(\frac{1}{r^2})$ as $r$ tends to $\infty$ and therefore
$$\partial_t(\rho\cos(\eta)) = \partial_t(\rho \sin(\eta))=0.$$
The deficit angle $\alpha$ can be expressed as
$$\alpha = 1-\lim_{r\rightarrow \infty}\frac{L(r)}{2\pi r},$$
where $L(r)$ is the length of the set of points which are at distance $r$ from the origin. In the coordinates where $g$ is asymptotic to
$$dr'^2+(1-\alpha)^2r'^2d\theta^2,$$
we see that we can write
$$\alpha = 1-\frac{1}{2\pi r}\lim_{r\rightarrow \infty} \int_0^{2\pi} \sqrt{g_{\theta \theta}}d\theta.$$
Since the evolution equation for $g$ is
$\partial_0 g_{ij}= -2NK_{ij},$
we have $\partial_t g_{\theta\theta}=O(r)$ as $r$ tends to $\infty$ and
$$\partial_t \alpha=0.$$
The deficit angle is said to be the two dimensional equivalent of the ADM energy, and we can naturally think of $\rho(\cos(\eta),\sin(\eta))$ as the ADM linear momentum.
\end{enumerate}

\subsection{Outline of the proof}
We will prove Theorem \ref{main} by a fixed point argument. The quantities $\dot{u}, \nabla u, \wht \tau$ and $b(\theta)$ are fixed.
\paragraph{The construction of the map F.}
We will construct a map
\begin{align*}
F:\m R \times H^2_\delta \times \q H^1_{\delta+1} &\rightarrow \m R \times H^2_\delta \times \q H^1_{\delta+1} \\
(\alpha, \widetilde{\lambda},\widetilde{H}) &\mapsto (\alpha', \widetilde{\lambda'},\widetilde{H'}),
\end{align*} 
such that $(\lambda', H')$ given by
\begin{align*}
\lambda' &= -\alpha'\ln(r)\chi(r) + \widetilde{\lambda'},\\
H' &= (b(\theta)+\rho\cos(\theta-\eta))H_\theta+ \widetilde{H'}, 
\end{align*}
are solutions of
\begin{align}
\label{eqhp}&\partial_i H'_{ij} +H_{ij}\partial_i \lambda  = -\dot{u}\partial_j u + \frac{1}{2} \partial_j \tau-\frac{1}{2} \tau\partial_j \lambda,\\
\label{eqlp}&\Delta \lambda' + \frac{1}{2}\dot{u}^2 + \frac{1}{2}|\nabla u|^2+\frac{1}{2}|H|^2-\frac{\tau^2}{4} = 0,
\end{align}
with $\lambda, H, \tau$ defined by
\begin{align}
\label{lam} \lambda &= -\alpha\ln(r)\chi(r) + \widetilde{\lambda},\\
\label{h}H &= (b(\theta)+\rho\cos(\theta-\eta))H_\theta+ \widetilde{H},\\
\label{tau}\tau &= \frac{b\chi(r)}{r} +\frac{\rho\chi(r)}{r}\cos(\theta-\eta) +\widetilde{\tau},
\end{align}
where $\rho, \eta$ depending on $\alpha, \widetilde{\lambda}, \widetilde{H} $ are constructed during the process in order to have the same parameters $\rho,\eta$ in $H,\tau$ and in $H',\tau'$.
Then, proving that $F$ has a fixed point easily follows from the estimates derived for $\alpha', \lambda'$ and $H'$, which concludes the proof of Theorem 
\ref{main}. Thus the core of the analysis is to solve \eqref{eqhp} and \eqref{eqlp}.

\paragraph{Solving \eqref{eqhp}.} 
For $\lambda, H, \tau$ of the form \eqref{lam}, \eqref{h} and \eqref{tau},
there always exists a solution of \eqref{eqhp} of the form
$$H'= \frac{\chi(r)}{r}\left(\begin{array}{cc}
g_1(\theta) & g_2(\theta) \\
g_2(\theta) & -g_1(\theta)
                      \end{array}\right)
+\widetilde{H'},$$
where $g_1$ and $g_2$ are two functions of the angle $\theta$, and $\widetilde{H'}$ is a traceless symmetric tensor belonging to $\q H^1_{\delta+1}$.
For $\varepsilon>0$ small enough, we will be able to choose $\rho$ and $\eta$ such that $H'$ can be written under the form $(b(\theta)+\rho\cos(\theta-\eta))H_\theta+ \widetilde{H'}$.

\paragraph{Solving \eqref{eqlp}.}
We remark that, according to Theorem \ref{laplacien}, the equation \eqref{eqlp} may not have solutions in $H^2_{\delta}$. 
Also, because of the asymptotic behavior of $|H|^2$ and $\tau^2$, which are only decreasing like $\frac{1}{r^2}$ as $r \rightarrow \infty$, 
the right-hand side of the equation \eqref{eqlp} may not be in the space $H^0_{\delta+2}$, in which case we may not be able
to apply Corollary \ref{coro} to solve the equation. However, the particular form of $H$ and $\tau$ allows the terms decreasing like $\frac{1}{r^2}$
to balance each other, and we are able to obtain a solution of \eqref{eqlp} of the form $ -\alpha'\ln(r)\chi(r) + \widetilde{\lambda'}$ with
$$\alpha' = \frac{1}{2\pi}\int \left( \frac{1}{2}\dot{u}^2 + \frac{1}{2}|\nabla u|^2+\frac{1}{2}|H|^2-\frac{\tau^2}{4}\right).$$
\\

The rest of the paper is as follows.
In section \ref{moment}, we explain how to solve the momentum constraint \eqref{eqhp}. In section \ref{lic}, we explain how
to choose the coefficients $\rho, \eta$ and how to solve the Lichnerowicz equation \eqref{eqlp}. Finally
the map $F$ is constructed in section \ref{concl}. It is shown to have a fixed point, which concludes the proof of Theorem \ref{main}.

\section{The momentum constraint}\label{moment}
The goal of this section is to solve equation \eqref{eqhp}.
\begin{prp}\label{solh}
We assume $\dot{u}.\nabla u \in H^0_{\delta+2}$. Let $b\in L^\infty(\m S^1)$ such that
$$\int_{\m S^1} b(\theta)\cos(\theta)d\theta=\int_{\m S^1} b(\theta)\sin(\theta)d\theta =0.$$
Let $\alpha , \rho, \eta \in \m R$, and let
\begin{align*}
 \tau&=b(\theta)\frac{\chi(r)}{r}+ \rho\frac{\chi(r)}{r}\cos(\theta-\eta)+\widetilde{\tau},\\
\lambda &=-\alpha \chi(r)\ln(r) + \widetilde{\lambda},\\
H&=(b(\theta)+\rho\cos(\theta-\eta))H_\theta + \widetilde{H},
\end{align*}
with
$\widetilde{\tau}\in H^1_{\delta+1}$, $\widetilde{\lambda} \in H^2_\delta$, $\widetilde{H} \in \q H^1_{\delta+1}$.
Then the equation
$$\partial_i H'_{ij} +H_{ij}\partial_i \lambda  = -\dot{u}.\partial_j u + \frac{1}{2} \partial_j \tau-\frac{1}{2} \tau\partial_j \lambda,$$
has a unique solution of the form
\begin{align*}
H' =&
\frac{m\chi(r)}{r}\left(\begin{array}{cc}
                                          \cos(\theta+\phi) & \sin(\theta+\phi) \\ 
                                           \sin(\theta+\phi) & -\cos(\theta+\phi)
                                         \end{array} \right)\\
&-\frac{\rho \chi(r)}{4r}\left(\begin{array}{cc}
                                          \cos(3\theta-\eta) & \sin(3\theta-\eta) \\ 
                                           \sin(3\theta-\eta) & -\cos(3\theta-\eta)
                                         \end{array} \right) +b(\theta)H_\theta + \widetilde{H}',
\end{align*}
 with $\widetilde{H}' \in \q H^1_{\delta+1}$ and
\begin{equation}\label{cos} 
\begin{split}
m\cos(\phi) = \frac{1}{2\pi}\int &\bigg(-\dot{u}.\partial_1 u -\frac{1}{2}\widetilde{\tau}\partial_1 \lambda
-\widetilde{H}_{i1}\partial_i \lambda
-\partial_i \widetilde{\lambda}\big((b(\theta)+\rho\cos(\theta-\eta))H_\theta\big)_{i1}\\
 &-\frac{1}{2}\chi(r)\frac{b(\theta)+\rho\cos(\theta-\eta)}{r}\partial_1 \widetilde{\lambda} \bigg)+\frac{ \rho}{4} \cos(\eta),
\end{split}
\end{equation}
\begin{equation}\label{sin} 
\begin{split}
m\sin(\phi) = \frac{1}{2\pi}\int &\bigg(-\dot{u}.\partial_2 u 
-\frac{1}{2}\widetilde{\tau}\partial_2 \lambda
-\widetilde{H}_{i2}\partial_i \lambda 
-\partial_i \widetilde{\lambda}\big((b(\theta)+\rho\cos(\theta-\eta))H_\theta\big)_{i2}\\
& -\frac{1}{2}\chi(r)\frac{b(\theta)+\rho\cos(\theta-\eta)}{r}\partial_2 \widetilde{\lambda}\bigg)
+ \frac{\rho}{4} \sin(\eta).
\end{split}
\end{equation}
Moreover we have
\begin{equation*}
 \begin{split}
\|\widetilde{H'}\|_{\q H^1_{\delta+1}} \lesssim &\|b\|_{L^\infty} + \|\dot{u}.\nabla u\|_{H^0_{\delta +2}} +(1+|\alpha|+\|\widetilde{\lambda}\|_{H^2_\delta})\|\widetilde{\tau}\|_{H^1_{\delta+1}}\\
 &+ |\rho| + (\|\widetilde{H}\|_{\q H^1_{\delta+1}} +\|b\|_{L^\infty}+|\rho|) \|\widetilde{\lambda}\|_{H^2_\delta} +|\alpha|\|\widetilde{H}\|_{\q H^1_{\delta+1}} .
\end{split}
\end{equation*}
\end{prp}

To prove this proposition, we write $H'= H^{(1)} + H^{(2)} + H^{(3)}$ with
\begin{align}
\label{h1}\partial_i H^{(1)}_{ij} =& -\dot{u}.\partial_j u +\frac{1}{2}\partial_j \widetilde{\tau} -\frac{1}{2}\widetilde{\tau}\partial_j \lambda
-\widetilde{H}_{ij}\partial_i \lambda \\
\nonumber & +\frac{\rho\chi'(r)}{4r}e_j  -\partial_i \widetilde{\lambda}\big((b(\theta)+\rho\cos(\theta-\eta))H_\theta\big)_{ij} -\frac{1}{2}\chi(r)\frac{b(\theta)+\rho\cos(\theta-\eta)}{r}\partial_j \widetilde{\lambda},\\
 \label{h2}\partial_i H^{(2)}_{ij} =&\frac{1}{2} \partial_j \left(\frac{b(\theta)\chi(r)}{r}\right)+ (b(\theta)H_\theta)_{ij} \partial_i (\alpha \chi(r)\ln(r)) +\frac{1}{2}\frac{b(\theta)\chi(r)}{r}\partial_j (\alpha \chi(r)\ln(r)),\\
\label{h3}\partial_i H^{(3)}_{ij} =& \frac{1}{2}\partial_j \left(\frac{\rho \cos(\theta-\eta)\chi(r)}{r}\right) 
-\frac{\rho\chi'(r)}{4r}e_j 
+(\rho\cos(\theta-\eta)H_\theta)_{ij} \partial_i \left(\alpha \chi(r)\ln(r)\right)\\
\nonumber &+\frac{1}{2}\frac{\rho\chi(r)\cos(\theta-\eta)}{r}\partial_j (\alpha \chi(r)\ln(r)),
\end{align}
where $e_1=\cos(\eta)$ and $e_2= \sin(\eta)$.
The following three propositions allow us to solve \eqref{h1}, \eqref{h2} and \eqref{h3}.

\begin{prp}\label{prph1} 
There exists a unique solution of \eqref{h1} of the form
$$H^{(1)} = \frac{m\chi(r)}{r}\left(\begin{array}{cc}
                                          \cos(\theta+\phi) & \sin(\theta+\phi) \\ 
                                           \sin(\theta+\phi) & -\cos(\theta+\phi)
                                         \end{array} \right)\\ +\widetilde{H}^{(1)},$$
with $\widetilde{H}^{(1)} \in H^1_{\delta +1}$ and
$m\cos(\phi)$ and $m\sin(\phi)$ are defined by \eqref{cos} and \eqref{sin}.
Moreover, $\widetilde{H}^{(1)}$ satisfies the estimate 
\begin{equation}\label{est}
 \begin{split}
\|\widetilde{H'}\|_{\q H^1_{\delta+1}} \lesssim &\|\dot{u}.\nabla u\|_{H^0_{\delta +2}} +(1+|\alpha|+\|\widetilde{\lambda}\|_{H^2_\delta})\|\widetilde{\tau}\|_{H^1_{\delta+1}}\\
 &+ |\rho| + (\|\widetilde{H}\|_{\q H^1_{\delta+1}} +\|b\|_{L^\infty}+|\rho|) \|\widetilde{\lambda}\|_{H^2_\delta} +|\alpha|\|\widetilde{H}\|_{\q H^1_{\delta+1}}.
\end{split}
\end{equation}
\end{prp}

\begin{prp}\label{prph2}
There exists an unique $\widetilde{H}^{(2)} \in \q H^1_{\delta+1}$ such that $H^{(2)} = b(\theta)H_\theta +\widetilde{H}^{(2)}$  satisfies \eqref{h2}.
Moreover we have the estimate
$$\left\|\widetilde{H}^{(2)}\right\|_{\q H^1_{\delta+1}} \lesssim \|b\|_{L^\infty}.$$
\end{prp}

\begin{prp}\label{prph3}
There exists an unique $\widetilde{H}^{(3)} \in \q H^1_{\delta+1}$ such that
$$H^{(3)}= -\frac{\rho \chi(r)}{4r}\left(\begin{array}{cc}
                                          \cos(3\theta-\eta) & \sin(3\theta-\eta) \\ 
                                           \sin(3\theta-\eta) & -\cos(3\theta-\eta)
                                         \end{array} \right)+\widetilde{H}^{(3)},$$
satisfies \eqref{h3}. Moreover, we have the estimate
$$\left\|\widetilde{H}^{(3)}\right\|_{\q H^1_{\delta+1}} \lesssim |\rho|.$$
\end{prp}

Since the sum of the right-hand sides of \eqref{h1}, \eqref{h2} and \eqref{h3} is equal to
$$-H_{ij}\partial_i \lambda -\dot{u}\partial_j u + \frac{1}{2} \partial_j \tau-\frac{1}{2} \tau\partial_j \lambda,$$
Proposition \ref{solh} is a straightforward consequence of Propositions \ref{prph1}, \ref{prph2} and \ref{prph3}.
Thus, in the rest of this section, we prove Proposition \ref{prph1}, \ref{prph2} and \ref{prph3}, respectively in section
\ref{sh1}, \ref{sh2} and \ref{sh3}.

\subsection{Proof of Proposition \ref{prph1}}\label{sh1}
We need the following lemma.
\begin{lm}\label{lemme} Let $f_1,f_2 \in H^0_{\delta+2}$. The equation
$$\partial_i K_{ij} = f_j,$$
with $K$ a symmetric traceless tensor, has a unique solution of the form
$$K = \frac{m\chi(r)}{r}\left(\begin{array}{cc}
                                          \cos(\theta+\phi) & \sin(\theta+\phi) \\ 
                                           \sin(\theta+\phi) & -\cos(\theta+\phi)
                                         \end{array} \right)\\ +\widetilde{K},$$
with 
$$m(\cos(\phi),\sin(\phi)) = \frac{1}{2\pi}\left(\int f_1, \int f_2\right)$$
 and $\widetilde{K} \in \q H^1_{\delta +1}$ with 
$$\|\widetilde{K} \|_{ H^1_{\delta +1}} \lesssim \|f_1\|_{H^0_{\delta+2}} + \|f_2\|_{H^0_{\delta+2}}.$$
\end{lm}

We postpone the proof of Lemma \ref{lemme} to the end of the section, and use it to prove Proposition \ref{prph1}.
\begin{proof}[Proof of Proposition \ref{prph1}]
We apply Lemma \ref{lemme} with
\begin{equation}\label{fj}
\begin{split}
  f_j = &-\dot{u}.\partial_j u +\frac{1}{2}\partial_j \widetilde{\tau} -\frac{1}{2}\widetilde{\tau}\partial_j \lambda
-\widetilde{H}_{ij}\partial_i \lambda  +\frac{\rho\chi'(r)}{4r}e_j \\
&-\partial_i \widetilde{\lambda}\big((b(\theta)+\rho \cos(\theta-\eta))H_\theta\big)_{ij} -\frac{1}{2}\chi(r)\frac{b+\rho\cos(\theta-\eta)}{r}\partial_j \widetilde{\lambda}, \quad j=1,2.\\ 
\end{split}
\end{equation}
We first check that $f_j$ belongs to $H^0_{\delta+2}$. 
Since $\widetilde{\tau} \in H^1_{\delta+1}$, we have $\partial_j \widetilde{\tau} \in H^0_{\delta+2}$ with
\begin{equation}\label{esta}
\|\partial_j \widetilde{\tau}\|_{H^0_{\delta+2}} \lesssim \|\widetilde{\tau}\|_{H^1_{\delta+1}}.
\end{equation}
Moreover, thanks to Lemma \ref{produit2}, we have $\frac{\chi(r)}{r}\widetilde{\tau} \in H^1_{\delta+2}$. Since
$\widetilde{\lambda} \in H^2_\delta$, we have $\partial_j \widetilde{\lambda} \in H^1_{\delta+1}$ and therefore,
thanks to Proposition \ref{produit}, $\widetilde{\tau}\partial_j \widetilde{\lambda} \in H^0_{\delta+2}$. Consequently we have the estimate
\begin{equation}\label{estb}
\|\widetilde{\tau}\partial_j \lambda\|_{H^0_{\delta+2}} \lesssim (|\alpha| + \|\widetilde{\lambda}\|_{H^2_\delta})\|\widetilde{\tau}\|_{H^1_{\delta+1}}.
\end{equation}
In the same way, we have the estimates
\begin{align}
 \label{estc}&\|\widetilde{H}_{ij}\partial_i \lambda\|_{H^0_{\delta+2}} \lesssim (|\alpha| + \|\widetilde{\lambda}\|_{H^2_\delta})\|\widetilde{H}\|_{\q H^1_{\delta+1}},\\
\label{estd}&\left\|\partial_i \widetilde{\lambda}((b(\theta)+\rho \cos(\theta-\eta))H_\theta)_{ij} -\frac{1}{2}\chi(r)\frac{b+\rho\cos(\theta-\eta)}{r}\partial_j \widetilde{\lambda}\right\|_{H^0_{\delta+2}}
\lesssim (\|b\|_{L^\infty}+ |\rho|)\|\widetilde{\lambda}\|_{H^2_\delta}.
\end{align}
Since $\chi'$ is compactly supported, we have 
\begin{equation}\label{este}
 \left\| \frac{\rho\chi'(r)}{4r}e_j\right\|_{H^0_{\delta+2}}\lesssim |\rho|.
\end{equation}
\eqref{esta}, \eqref{estb}, \eqref{estc}, \eqref{estd} and \eqref{este} yield
\begin{equation}\label{estf2}
 \begin{split}
\|f_1\|_{H^0_{\delta+2}} + \|f_2\|_{H^0_{\delta+ 2}}
\lesssim& \|\dot{u}.\nabla u\|_{H^0_{\delta +2}} +(1+|\alpha|+\|\widetilde{\lambda}\|_{H^2_\delta})\|\widetilde{\tau}\|_{H^1_{\delta+1}}\\
 &+ |\rho| + (\|\widetilde{H}\|_{\q H^1_{\delta+1}} +\|b\|_{L^\infty}+|\rho|) \|\widetilde{\lambda}\|_{H^2_\delta} +|\alpha|\|\widetilde{H}\|_{\q H^1_{\delta+1}}.
\end{split}
\end{equation}
Therefore, Lemma \ref{lemme} implies that we have a unique solution of \eqref{h1} of the form
$$H^{(1)} = \frac{m\chi(r)}{r}\left(\begin{array}{cc}
                                          \cos(\theta+\phi) & \sin(\theta+\phi) \\ 
                                           \sin(\theta+\phi) & -\cos(\theta+\phi)
                                         \end{array} \right)\\ +\widetilde{H}^{(1)},$$
with $\widetilde{H}^{(1)} \in \q H^1_{\delta+1}$. Together with \eqref{estf2}, it yields the estimate \eqref{est}.
Then, in view of the definition \eqref{fj} of $f_1$ and $f_2$, the computations 
\begin{align*}
 \int  \frac{\rho\chi'(r)}{4r}\cos(\eta)rdrd\theta& =  \frac{\pi\rho}{2}  \cos(\eta),\\
 \int  \frac{\rho\chi'(r)}{4r}\sin(\eta)rdrd\theta& =  \frac{\pi\rho}{2}  \sin(\eta),\\
\end{align*}
and the fact that 
$$\int \partial_j \tilde{\tau} = 0,$$ 
yield the identities \eqref{cos} and \eqref{sin}.
\end{proof}

\begin{proof}[Proof of lemma \ref{lemme}]
We look for solutions of the form 
\begin{equation}\label{formeh} K_{ij} = \partial_i Y_j + \partial_j Y_i-\delta_{ij} \partial^k Y_k.
\end{equation}
The vector $Y$ then satisfies the equations
\begin{align*}
 \Delta Y_1 &= f_1,\\
\Delta Y_2 &= f_2.
\end{align*}

We can apply Corollary \ref{coro} which says that 
$$Y_j = \frac{1}{2\pi}\left(\int f_j\right)\chi(r)\ln(r) + \widetilde{Y_j}$$ 
where $\widetilde{Y_j} \in H^2_{\delta}$ satisfies 
$$\|\widetilde{Y_j} \|_{H^2_{\delta}} \lesssim \|f_j\|_{H^0_{\delta+2}}.$$ We have then
\begin{align*}
 K_{11} &= \partial_1 Y_1 - \partial_2 Y_2 = \chi(r)\frac{x_1\int f_1-x_2\int f_2}{2\pi r^2} + \widetilde{K}_{11},\\
K_{12} &= \partial_1 Y_2 + \partial_2 Y_1 = \chi(r)\frac{x_1\int f_2 +x_2\int f_1}{2\pi r^2} + \widetilde{K}_{12},
\end{align*}
where $ \widetilde{K}_{11},  \widetilde{K}_{12} \in H^1_{\delta + 1}$ satisfy
$$\|\widetilde{K}\|_{H^1_{\delta + 1}} \lesssim \|f_1\|_{H^0_{\delta + 2}} +\|f_2\|_{H^0_{\delta + 2}}.$$
Let 
$$m(\cos \phi, \sin\phi) = \frac{1}{2\pi}\left(\int f_1, \int f_2\right).$$  We obtain
\begin{align*}
\frac{x_1\int f_1-x_2\int f_2}{2\pi r^2} &=\frac{m}{r}(\cos\phi \cos \theta-\sin\phi \sin \theta) = \frac{m}{r}\cos(\theta + \phi),\\
\frac{x_1\int f_2 +x_2\int f_1}{2\pi r^2} &=\frac{m}{r}(\sin\phi \cos \theta + \cos\phi \sin \theta) = \frac{m}{r}\sin(\theta + \phi).
\end{align*}
This yields
\begin{align*}
 K_{11} &= \chi(r)\frac{m\cos(\theta+\phi)}{r} +\widetilde{K}_{11},\\
 K_{12} &= \chi(r)\frac{m\sin(\theta+\phi)}{r} +\widetilde{K}_{12}.
\end{align*}
Since $K$ is symmetric and traceless, we have $K_{22} = -K_{11}$ and $K_{12}=K_{21}$.
Moreover, $K$ is unique, because if $\widetilde{H} \in \q H^1_{\delta +1}$ satisfies $D^i \widetilde{H}_{ij}= 0$, then
$\Delta \widetilde{H}_{ij} = 0$ for $i,j= 1,2$, and therefore $\widetilde{H}=0$.
This concludes the proof of Lemma
\ref{lemme}. 
\end{proof}
We point out that the method we have used to solve $\partial_i K_{ij}=f_j$ is the one used in the conformal method. In the conformal method, one looks for solutions of the form
$$K=LY +\sigma,$$
where $L$ is the conformal Killing operator (defined in our case by \eqref{formeh}) and $\sigma$ a transverse traceless tensor. In our case, since there are no transverse traceless tensors which decay at infinity on $\m R^2$, we can pick $\sigma=0$.

\subsection{Proof of Proposition \ref{prph2}}\label{sh2}

 We calculate the right-hand side of \eqref{h2} for $j=1$
\begin{align*}
&\frac{1}{2} \partial_1 \left(\frac{b(\theta)\chi(r)}{r}\right)+ (b(\theta)H_\theta)_{i1} \partial_i (\alpha \chi(r)\ln(r)) +\frac{1}{2}\frac{b(\theta)\chi(r)}{r}\partial_1 (\alpha \chi(r)\ln(r))\\
=&\frac{b(\theta)}{2}\cos(\theta)\left(-\frac{\chi(r)}{r^2}+\frac{\chi'(r)}{r}\right)
-\alpha \frac{b(\theta)\chi(r)}{2r}\left(\chi'(r)\ln(r)+\frac{\chi(r)}{r}\right)(\cos(\theta)\cos(2\theta)+\sin(\theta)\sin(2\theta))\\
&+\alpha\frac{b(\theta)\chi(r)}{2r}\left(\chi'(r)\ln(r)+\frac{\chi(r)}{r}\right)\cos(\theta)-\frac{\sin(\theta)
\partial_\theta b(\theta)}{2}\frac{\chi(r)}{r^2}\\
=&\frac{b}{2}\cos(\theta)\left(-\frac{\chi(r)}{r^2}+\frac{\chi'(r)}{r}\right)
-\frac{\sin(\theta)\partial_\theta b(\theta)}{2}\frac{\chi(r)}{r^2}.
\end{align*}
We have similarly for $j=2$,
\begin{align*}
&\frac{1}{2} \partial_2 \left(\frac{b(\theta)\chi(r)}{r}\right)+ (b(\theta)H_\theta)_{i2} \partial_i (\alpha \chi(r)\ln(r)) +\frac{1}{2}\frac{b(\theta)\chi(r)}{r}\partial_2 (\alpha \chi(r)\ln(r))\\
=&\frac{b(\theta)}{2}\sin(\theta)\left(-\frac{\chi(r)}{r^2}+\frac{\chi'(r)}{r}\right)
+\frac{\cos(\theta)\partial_\theta b(\theta)}{2}\frac{\chi(r)}{r^2}.
\end{align*}
We calculate then,
\begin{align*}
 \partial_i (b(\theta)H_\theta)_{i1} =& -\frac{2b(\theta)\chi(r)}{2r^2}(\sin(\theta)\sin(2\theta)+\cos(\theta)\cos(2\theta))\\
&-\frac{b(\theta)}{2}\left(-\frac{\chi(r)}{r^2} + \frac{\chi'(r)}{r^2}\right)(\cos(\theta)\cos(2\theta) + \sin(\theta)\sin(2\theta))\\
&-\frac{\partial_\theta b(\theta)\chi(r)}{2r^2}(-\sin(\theta)\cos(2\theta)+\cos(\theta)\sin(2\theta))\\
=&-\frac{b(\theta)\chi(r)}{2r^2}\cos(\theta)-\frac{b(\theta)\chi'(r)}{2r}\cos(\theta)
-\frac{\sin(\theta)\partial_\theta b(\theta)}{2}\frac{\chi(r)}{r^2}.
\end{align*}
In the same way
$$ \partial_i (b(\theta)H_\theta)_{i2}=-\frac{b(\theta)\chi(r)}{2r^2}\sin(\theta)-\frac{b(\theta)\chi'(r)}{2r}\sin(\theta)
+\frac{\cos(\theta)\partial_\theta b(\theta)}{2}\frac{\chi(r)}{r^2}.$$
Therefore $b(\theta)H_\theta + \widetilde{H}^{(2)}$ satisfies \eqref{h2} if and only if
\begin{equation}\label{th2}
\partial_i \widetilde{H}^{(2)}_{ij} = f_j,
\end{equation}
with
$$\left(\begin{array}{l}
   f_1\\ f_2
  \end{array}\right)
=\frac{b(\theta)\chi'(r)}{r}\left(\begin{array}{l}
   \cos(\theta) \\ \sin(\theta)
  \end{array}\right).$$
We have $f_1,f_2 \in H^0_{\delta+2}$ with
$$\|f_1\|_{H^0_{\delta+2}} +\|f_2\|_{H^0_{\delta+2}} \lesssim \|b\|_{L^\infty},$$
and, since we assume 
$$\int b(\theta)\cos(\theta)\theta = \int b(\theta)\sin(\theta)d\theta =0,$$
we have also
$$ \int f_1 = \int f_2= 0.$$
Therefore Lemma \ref{lemme} implies that there exists a unique solution $\widetilde{H}^{(2)}\in H^1_{\delta+1}$ of \eqref{th2}. Furthermore it satisfies the estimate
$$\left\|\widetilde{H}^{(2)}\right\|_{\q H^1_{\delta+1}} \lesssim \|b\|_{L^\infty},$$
which concludes the proof of Proposition \ref{prph2}.

\subsection{Proof of Proposition \ref{prph3}}\label{sh3}

We calculate each term of the right-hand side of \eqref{h3} for $j=1$
\begin{equation*}
 \begin{split}
 & \frac{1}{2} \partial_1 \left(\frac{\rho \cos(\theta-\eta)\chi(r)}{r}\right)-\frac{\rho\chi'(r)}{4r}\cos(\eta)\\
&= \frac{\rho\chi(r)}{2r^2}\left(\sin(\theta)\sin(\theta-\eta) -\cos(\theta)\cos(\theta-\eta)\right)+ 
\frac{\rho\chi'(r)}{2r}\cos(\theta)\cos(\theta-\eta)-\frac{\rho\chi'(r)}{4r}\cos(\eta)\\
&=-\frac{\rho\chi(r)}{2r^2}\cos(2\theta-\eta) +\frac{\rho\chi'(r)}{2r}\frac{\cos(\eta) + \cos(2\theta-\eta)}{2}-\frac{\rho\chi'(r)}{4r}\cos(\eta)\\
&=-\frac{\rho\chi(r)}{2r^2}\cos(2\theta-\eta) +\frac{\rho\chi'(r)}{4r}\cos(2\theta-\eta),
 \end{split}
\end{equation*}
\begin{equation*}
 \begin{split}
&(H_{\rho,\eta})_{i1} \partial_i (\alpha \chi(r)\ln(r))
 +\frac{1}{2}\frac{\rho\chi(r)\cos(\theta-\eta)}{r}\partial_1 (\alpha \chi(r)\ln(r))\\
=&  -\frac{\alpha \rho\chi(r)}{4r}\left(\partial_r (\chi(r)\ln(r))\right)\left(\cos(\theta)(\cos(\theta+\eta)+\cos(3\theta - \eta))
+\sin(\theta)(\sin(\theta+ \eta) + \sin(3\theta- \eta))\right)\\
&+ \frac{\alpha \rho \chi(r)}{2r}\left(\partial_r (\chi(r)\ln(r))\right)\cos(\theta)\cos(\theta-\eta)\\
=&-\frac{\alpha \rho\chi(r)}{4r}\left(\partial_r (\chi(r)\ln(r))\right)(\cos(\eta)+\cos(2\theta- \eta))
+ \frac{\alpha \rho \chi(r)}{2r}\left(\partial_r (\chi(r)\ln(r))\right)\frac{\cos(\eta) + \cos(2\theta-\eta)}{2}\\
=&0.
 \end{split}
\end{equation*}
Therefore we have
\begin{align*}
&\frac{1}{2} \partial_1 \left(\frac{\rho \cos(\theta-\eta)\chi(r)}{r}\right)-\frac{\rho\chi'(r)}{4r}\cos(\eta) +(H_{\rho,\eta})_{i1} \partial_i (\alpha \chi(r)\ln(r))\\
 &+\frac{1}{2}\frac{\rho\chi(r)\cos(\theta-\eta)}{r}\partial_1 (\alpha \chi(r)\ln(r))\\
=&-\frac{\rho\chi(r)}{2r^2}\cos(2\theta-\eta)+\frac{\rho\chi'(r)}{4r}\cos(2\theta-\eta).
\end{align*}
In the same way, for $j=2$ we have
\begin{align*}
&\frac{1}{2} \partial_2 \left(\frac{\rho \cos(\theta-\eta)\chi(r)}{r}\right)-\frac{\rho\chi'(r)}{4r}\sin(\eta) +(H_{\rho,\eta})_{i2} \partial_i (\alpha \chi(r)\ln(r))\\
 &+\frac{1}{2}\frac{\rho\chi(r)\cos(\theta-\eta)}{r}\partial_2 (\alpha \chi(r)\ln(r))\\
=&-\frac{\rho\chi(r)}{2r^2}\sin(2\theta-\eta)+\frac{\rho\chi'(r)}{4r}\sin(2\theta-\eta).
\end{align*}
We calculate now
\begin{align*}
 &\partial_1 \left(-\frac{\rho\chi(r)\cos(3\theta-\eta)}{4r}\right) + \partial_2 \left( -\frac{\rho\chi(r)\sin(3\theta-\eta)}{4r}\right)\\
=& -\frac{\rho}{4}(\cos(\theta) \cos(3\theta-\eta) + \sin(\theta) \sin(3\theta-\eta))\left(-\frac{\chi(r)}{r^2} + \frac{\chi'(r)}{r}\right)\\
&-3\frac{\rho}{4}\left(\sin(\theta)\sin(3\theta-\eta) + \cos(\theta) \cos(3\theta-\eta)\right)\frac{\chi(r)}{r^2}\\
=&-\frac{\rho\chi(r)}{2r^2}\cos(2\theta-\eta) - \frac{\rho \chi'(r)}{4r} \cos(2\theta-\eta). 
\end{align*}
In the same way we have
\begin{align*}
 &-\partial_2 \left(-\frac{\rho\chi(r)\cos(3\theta-\eta)}{4r}\right) + \partial_1 \left( -\frac{\rho\chi(r)\sin(3\theta-\eta)}{4r}\right)\\
=&-\frac{\rho\chi(r)}{2r^2}\sin(2\theta-\eta) - \frac{\rho \chi'(r)}{4r} \sin(2\theta-\eta). 
\end{align*}
Therefore
$$H^{(3)}= -\frac{\rho \chi(r)}{4r}\left(\begin{array}{cc}
                                          \cos(3\theta-\eta) & \sin(3\theta-\eta) \\ 
                                           \sin(3\theta-\eta) & -\cos(3\theta-\eta)
                                         \end{array} \right) + \widetilde{H}^{(3)}$$
satisfies \eqref{h3} if and only if
\begin{equation}\label{th3}
\partial_i \widetilde{H}^{(3)}_{ij} = f_j,
\end{equation}
with
$$\left(\begin{array}{l}
   f_1\\ f_2
  \end{array}\right)
=\frac{\rho\chi'(r)}{2r}\left(\begin{array}{l}
   \cos(2\theta-\eta) \\ \sin(2\theta-\eta)
  \end{array}\right).$$
We have $f_1,f_2 \in H^0_{\delta+2}$ with
$$\|f_1\|_{H^0_{\delta+2}} +\|f_2\|_{H^0_{\delta+2}} \lesssim |\rho|,$$
and
$$ \int f_1 = \int f_2= 0.$$
Therefore Lemma \ref{lemme} implies that there exists a unique solution $\widetilde{H}^{(3)}\in H^1_{\delta+1}$ of \eqref{th3}. Furthermore, it satisfies the estimate
$$\left\|\widetilde{H}^{(3)}\right\|_{\q H^1_{\delta+1}} \lesssim |\rho|,$$
which concludes the proof of Proposition \ref{prph3}.

\section{The choice of $\rho,\eta$ and the Lichnerowicz equation}\label{lic}
The goal of this section is to solve equation \eqref{eqlp}.

\subsection{The choice of $\rho,\eta$}
We assume $\dot{u}\nabla u \in H^2_\delta$. Let $\alpha \in \m R$, $b \in L^\infty(\m S^1)$, $\widetilde{\lambda} \in H^2_\delta$ and $\widetilde{H}\in \q H^1_{\delta+1}$. We
consider the map 
\begin{align*}
 G : \m R\times \m S^1 &\rightarrow \m R\times \m S^1\\
(\rho, \eta)& \mapsto (-4m,\phi), 
\end{align*}
with $(m,\phi)$ given by
\begin{align*}
 m\cos(\phi) =& \frac{1}{2\pi}\int \bigg(-\dot{u}.\partial_1 u -\frac{1}{2}\widetilde{\tau}\partial_1 \lambda
-\widetilde{H}_{i1}\partial_i \lambda
-\partial_i \widetilde{\lambda}\big((b(\theta)+\rho\cos(\theta-\eta))H_\theta\big)_{i1}\\ &-\frac{1}{2}\chi(r)\frac{b(\theta)+\rho\cos(\theta-\eta)}{r}\partial_1 \widetilde{\lambda} \bigg)
+\frac{ \rho}{4} \cos(\eta),\\
 m\sin(\phi) = &\frac{1}{2\pi}\int \bigg(-\dot{u}.\partial_2 u 
-\frac{1}{2}\widetilde{\tau}\partial_2 \lambda
-\widetilde{H}_{i2}\partial_i \lambda 
-\partial_i \widetilde{\lambda}\big((b(\theta)+\rho\cos(\theta-\eta))H_\theta\big)_{i2}\\ &-\frac{1}{2}\chi(r)\frac{b(\theta)+\rho\cos(\theta-\eta)}{r}\partial_2 \widetilde{\lambda}\bigg)
+ \frac{ \rho}{4} \sin(\eta).
\end{align*}
We want our solution $H'$ of equation \eqref{eqhp} to have the same form as $H$, in order to find a fixed point $H= H'$ and
$\lambda = \lambda'$. Therefore, we need to show that $G$ has a fixed point. This is done in the following lemma.
\begin{lm}\label{rho}
 If $\|\widetilde{\lambda}\|_{H^2_\delta}\lesssim \varepsilon$, then for $\varepsilon>0$ small enough, $G$ admits an unique fixed point $(\rho,\eta)\in\m R\times [0,2\pi[ $.
Moreover, if we assume
\begin{equation}\label{hypo}\|\dot{u}\nabla u\|_{H^0_{\delta+2}} + \|\widetilde{H}\|_{\q H^1_{\delta+1}} + \|\widetilde{\tau}\|_{H^1_{\delta+1}} +\|\widetilde{\lambda}\|_{H^2_\delta}
+\|b\|_{L^\infty}+ |\alpha| \lesssim \varepsilon,
\end{equation}
then we have
\begin{align*}
 \rho \cos(\eta) &= \frac{1}{\pi}\int \dot{u}.\partial_1 u + O(\varepsilon ^{2}) ,\\
\rho \sin(\eta) &=  \frac{1}{\pi}\int \dot{u}.\partial_2 u + O(\varepsilon ^{2}).
\end{align*}
\end{lm}
\begin{proof}
 We write
$$a_j =\frac{1}{2\pi} \int \left(-\dot{u}.\partial_j u -\frac{1}{2}\widetilde{\tau}\partial_j \lambda
-\widetilde{H}_{ij}\partial_i \lambda
-\partial_i \widetilde{\lambda}(b(\theta)H_\theta)_{ij} -\frac{1}{2}\chi(r)\frac{b}{r}\partial_j \widetilde{\lambda} \right).$$
The conditions $\rho=-4m$ and $\eta=\phi$ are satisfied if and only if
\begin{align}
\label{rho1} \rho\cos(\eta)&= -4\left( a_1 +
\frac{1}{2\pi}\int\left(-\partial_i \widetilde{\lambda}(\rho\cos(\theta-\eta)H_\theta)_{i1} -\frac{1}{2}\chi(r)\frac{\rho\cos(\theta-\eta)}{r}\partial_1 \widetilde{\lambda}\right)
  +\frac{\rho}{4} \cos(\eta)\right),\\
\label{rho2} \rho\sin(\eta)&= -4\left(a_2
+\frac{1}{2\pi}\int\left(-\partial_i \widetilde{\lambda}(\rho\cos(\theta-\eta)H_\theta)_{i2} -\frac{1}{2}\chi(r)\frac{\rho\cos(\theta-\eta)}{r}\partial_2 \widetilde{\lambda}\right)
 +\frac{ \rho}{4} \sin(\eta)\right).
\end{align}
Since we assume $\|\widetilde{\lambda}\|_{H^2_\delta}\lesssim \varepsilon$, we can write this system under the form
$$\left(\begin{array}{cc}
         1 +O(\varepsilon) & O(\varepsilon)\\
	 O(\varepsilon) & 1+O(\varepsilon)
        \end{array}\right) \left( \begin{array}{c} \rho\cos(\eta) \\ \rho \sin(\eta) \end{array}\right)
 = -2\left(\begin{array}{c} a_1\\ a_2 \end{array}\right) .$$
For $\varepsilon>0$ small enough, it is invertible, so $G$ has a fixed point,  and we obtain, under
the hypothesis \eqref{hypo}
\begin{align*}
 \rho \cos(\eta) &= \frac{1}{\pi}\int \dot{u}\partial_1 u + O(\varepsilon ^{2}) ,\\
\rho \sin(\eta) &=  \frac{1}{\pi}\int \dot{u}\partial_2 u + O(\varepsilon ^{2}),
\end{align*}
which concludes the proof of Lemma \ref{rho}.
\end{proof}

\subsection{The Lichnerowicz equation}

\begin{prp}\label{lichn}
 Let $\dot{u}^2,|\nabla{u}|^2 \in H^0_{\delta+2}$. Let $b\in L^\infty(\m S^1)$, $\widetilde{\tau} \in H^1_{\delta +1}$, $\widetilde{H} \in \q H^1_{\delta+1}$, $\rho, \eta \in \m R$ and 
\begin{align*}
 H &= (b(\theta)+\rho\cos(\theta-\eta) + \widetilde{H},\\
\tau &= \frac{b(\theta)\chi(r)}{r} +\frac{\rho\chi(r)}{r}\cos(\theta-\eta) +\widetilde{\tau}.
\end{align*}
There exists a unique $\lambda'$ of the form 
$$\lambda' = -\alpha'\chi(r)\ln(r) + \widetilde{\lambda}',$$
with $\widetilde{\lambda'} \in H^2_{\delta}$, solution of
\begin{equation}\label{lambda}
\Delta \lambda' =- \frac{1}{2}\dot{u}^2 - \frac{1}{2}|\nabla u|^2-\frac{1}{2}|H|^2+\frac{\tau^2}{4}.
\end{equation}
Moreover, if 
\begin{equation}\label{est2}
\|\dot{u}^2\|_{H^0_{\delta+2}} + \||\nabla u|^2\|_{H^0_{\delta+2}} + \|\widetilde{H}\|_{\q H^1_{\delta+1}} + \|\widetilde{\tau}\|_{H^1_{\delta+1}} 
+\|b\|_{L^\infty}+ |\rho| \lesssim \varepsilon,
\end{equation}
we have 
$$\|\widetilde{\lambda'}\|_{H^2_\delta} \lesssim\|\dot{u}^2+|\nabla u|^2\|_{H^0_{\delta+2}} +\varepsilon^2$$
and
$$ \alpha' = \frac{1}{4\pi} \int \left(\dot{u}^2 + |\nabla u|^2\right) + O(\varepsilon^2).$$
\end{prp}

\begin{proof}
We want to apply Corollary \ref{coro}. We have to check that the right-hand side of \eqref{lambda} is in $H^0_{\delta+2}$.
We write
\begin{equation}\label{hettau}
\begin{split} 
|H|^2 &=|(b(\theta)+\rho\cos(\theta-\eta))H_\theta|^2 + f_1,\\
\tau^2 &= \left(\frac{(b+\rho\cos(\theta-\eta))\chi(r)}{r}\right)^2+f_2. 
\end{split}
\end{equation}

We first estimate $f_1$ and $f_2$.
Since $\widetilde{\tau} \in H^1_{\delta+1}$, $\widetilde{H} \in \q H^1_{\delta+1}$, we have thanks to Proposition \ref{produit} that $\widetilde{\tau}^2 ,|\widetilde{H}|^2
 \in H^0_{\delta+2}$, and thanks to Lemma \ref{produit2}, 
$$\widetilde{\tau}\left(\frac{(b+\rho\cos(\theta-\eta))\chi(r)}{r}\right) \in H^0_{\delta+2}, \quad
\widetilde{H}^{ij}((b(\theta)+\rho\cos(\theta-\eta))H_\theta)_{ij} \in H^0_{\delta+2}.$$
Therefore, we have $f_1, f_2 \in H^0_{\delta+2}$ with
\begin{equation}\label{estf}
 \|f_1\|_{H^0_{\delta+2}} + \|f_2\|_{H^0_{\delta+2}} \lesssim \|\widetilde{H}\|^2_{\q H^1_{\delta+1}}+ 
\|\widetilde{\tau}\|^2_{ H^1_{\delta+1}} +\|b\|_{L^\infty}^2 + \rho^2\lesssim \varepsilon^2.
\end{equation}

We now compute $\frac{1}{2}|(b(\theta)+\rho\cos(\theta-\eta))H_\theta|^2$ and 
$\frac{1}{4}\left(\frac{(b+\rho\cos(\theta-\eta))\chi(r)}{r}\right)^2$. Since we have
$$|H_\theta|^2 =\frac{2\chi(r)^2}{4r^2}(\cos(2\theta)^2+\sin(2\theta)^2)= \frac{\chi(r)^2}{2r^2},$$
we obtain
\begin{equation}\label{eg}
\frac{1}{2} |b(\theta)+\rho\cos(\theta-\eta))H_\theta|^2 = \frac{1}{4}\left(\frac{(b+\rho\cos(\theta-\eta))\chi(r)}{r}\right)^2.
\end{equation}

\eqref{hettau}, \eqref{estf} and \eqref{eg} imply that 
the right-hand side of \eqref{lambda} is in $H^0_{\delta+2}$ 
with
\begin{equation*}
\left\|- \frac{1}{2}\dot{u}^2 - \frac{1}{2}|\nabla u|^2-\frac{1}{2}|H|^2+\frac{\tau^2}{4}\right\|_{H^0_{\delta+2}}
\lesssim \|\dot{u}^2\|_{H^0_{\delta+2}} + \||\nabla u|^2\|_{H^0_{\delta+2}}
+\varepsilon^2, 
\end{equation*}
so Corollary \ref{coro} gives a unique solution of \eqref{lambda} of the form
$$\lambda' = -\alpha' \chi(r)\ln(r) + \widetilde{\lambda'},$$
with
$$\alpha' = \frac{1}{2\pi}\int \left( \frac{1}{2}\dot{u}^2 +\frac{1}{2}|\nabla u|^2+\frac{1}{2}|H|^2-\frac{\tau^2}{4}\right),$$
and $\widetilde{\lambda'} \in H^2_{\delta}$ .
If we assume \eqref{est2}, we obtain 
$$\alpha'=  \frac{1}{4\pi} \int \left(\dot{u}^2 + |\nabla u|^2\right) + O\left(\|f_1\|_{H^0_{\delta+2}} + \|f_2\|_{H^0_{\delta+2}}\right)
 =\frac{1}{4\pi} \int \left(\dot{u}^2 + |\nabla u|^2\right) + O(\varepsilon^2),$$
which, together with the estimate
$$\|\widetilde{\lambda'}\|_{H^2_{\delta}}\lesssim\|\dot{u}^2\|_{H^0_{\delta+2}} + \||\nabla u|^2\|_{H^0_{\delta+2}}+
\varepsilon^2,$$
concludes the proof of Proposition \ref{lichn}.
\end{proof}

\section{Proof of Theorem \ref{main}}\label{concl}
We note $X$ the Banach space $\m R \times H^2_\delta \times \q H^1_{\delta+1}$
equipped with the norm
$$\|(\alpha,\wht \lambda, \wht H)\|_{X}=|\alpha|+\|\wht \lambda\|_{H^2_\delta}
+\|\wht H\|_{\q H^1_{\delta+1}}.$$
We now have constructed, for $\varepsilon>0$ small enough, a map
$F:X\rightarrow X$
which maps $(\alpha, \widetilde{\lambda},\widetilde{H})$ satisfying
$$\|(\alpha,\wht \lambda, \wht H)\|_{X} \lesssim \varepsilon$$
to $(\alpha', \widetilde{\lambda'},\widetilde{H'})$ such that, if 
 $\rho, \eta$, depending on $\dot{u}, \nabla u , \widetilde{\tau}, \alpha, \lambda, \widetilde{H},b$, are given by Lemma \ref{rho},
 then 
$$H'=b(\theta)+\rho\cos(\theta-\eta))H_\theta + \widetilde{H}'$$
is the solution of
$$\partial_i H'_{ij} +H_{ij}\partial_i \lambda  = -\dot{u}.\partial_j u + \frac{1}{2} \partial_j \tau-\frac{1}{2} \tau\partial_j \lambda,$$
given by Proposition \ref{solh}, with
\begin{align*}
 H=& (b(\theta)+\rho\cos(\theta-\eta))H_\theta+ \widetilde{H}, \\
\lambda= &-\alpha\chi(r)\ln(r) + \widetilde{\lambda},\\
\tau = &\frac{\chi(r)}{r}(b(\theta)+\rho\cos(\theta-\eta))+\widetilde{\tau},
\end{align*}
and 
$$\lambda'= -\alpha'\chi(r)\ln(r) + \widetilde{\lambda'}$$
is the solution of
$$\Delta \lambda' =- \frac{1}{2}\dot{u}^2 - \frac{1}{2}|\nabla u|^2-\frac{1}{2}|H|^2+\frac{\tau^2}{4}$$
given by Proposition \ref{lichn}. Lemma \ref{rho} implies $|\rho|\lesssim \varepsilon$, and therefore Proposition \ref{solh} and Proposition \ref{lichn} imply
$$|\alpha'| + \| \widetilde{\lambda'}\|_{H^2_\delta} + \| \widetilde{H'}\|_{\q H^1_{\delta+1}} \lesssim \varepsilon.$$
In particular 
$\|F((0,0,0))\|_ {X} \lesssim \varepsilon,$
so there exists a constant $C$ such that 
$$\|F((0,0,0))\|_ {X}=C\varepsilon.$$\\

Next, we show that $F$ is a contraction map in  $B_{X }(0,2C\varepsilon)$.
We take, for $i=1,2$,  $(\alpha_i,\widetilde{\lambda}_i, \widetilde{H}_i)\in\m R \times H^2_\delta \times \q H^1_{\delta+1} $ such that
$$|\alpha_i|+\| \widetilde{\lambda}_i\|_{H^2_\delta} + \| \widetilde{H}_i\|_{\q H^1_{\delta+1}} \leq 2C \varepsilon.$$
We note $\rho_i$, $\eta_i$ the corresponding quantities given by Lemma \ref{rho}.
Thanks to formula \eqref{rho1} and \eqref{rho2} we get
\begin{align*}
&\rho_1\cos(\eta_1)-\rho_2\cos(\eta_2)\\
=&-\frac{1}{\pi}\int\left(-\frac{\wht \tau}{2}(\partial_1\lambda_1-\partial_1 \lambda_2)-(\wht H_1)_{i1}(\partial_i \lambda_1-\partial_i \lambda_2)+\partial_i \lambda_2((\wht H_1)_{i1}-(\wht H_2)_{i1})\right)\\
&-\frac{1}{\pi}\int\left(-(b(\theta)H_\theta)_{i1}(\partial_i \wht \lambda_1-\partial_i \wht \lambda_2)
-\frac{\chi(r)b}{2r}(\partial_1 \wht \lambda_1-\partial_1 \wht \lambda_2)
\right)\\
&-\frac{1}{\pi}\int -\left(\partial_i \wht \lambda_1(H_\theta)_{i1}+\frac{\chi(r)}{2}\partial_1 \wht \lambda_1\right)(\rho_1\cos(\theta-\eta_1)-\rho_2\cos(\theta-\eta_2))\\
&-\frac{1}{\pi}\int\rho_2\cos(\theta-\eta_2)\left((H_\theta)_{i1}(\partial_i \wht \lambda_1-\partial_i \wht \lambda_2)+\frac{\chi(r)}{2}(\partial_1 \wht \lambda_1-\partial_1 \wht \lambda_2)\right),
\end{align*}
and a similar formula for $\rho_1\sin(\eta_1)-\rho_2\sin(\eta_2)$. Lemma \ref{rho} implies that $|\rho_i|\lesssim \varepsilon$. Therefore we can write
\begin{align*}
|\rho_1\cos(\eta_1)-\rho_2 \cos(\eta_2)|\lesssim &\varepsilon\left(|\alpha_1-\alpha_2|  + \|\widetilde{\lambda}_1-\widetilde{\lambda}_2\|_{H^2_{\delta}}
+\|\widetilde{H}_1-\widetilde{H}_2\|_{\q H^1_{\delta+1}}\right) \\
& + \varepsilon\left( |\rho_1\cos(\eta_1)-\rho_2 \cos(\eta_2)| +|\rho_1\sin(\eta_1)-\rho_2 \sin(\eta_2)|\right),\\
|\rho_1\sin(\eta_1)-\rho_2 \sin(\eta_2)|\lesssim  &\varepsilon\left(|\alpha_1-\alpha_2| + \|\widetilde{\lambda}_1-\widetilde{\lambda}_2\|_{H^2_{\delta}}
+\|\widetilde{H}_1-\widetilde{H}_2\|_{\q H^1_{\delta+1}}\right) \\
&+ \varepsilon\left( |\rho_1\cos(\eta_1)-\rho_2 \cos(\eta_2)| +|\rho_1\sin(\eta_1)-\rho_2 \sin(\eta_2)|\right),
\end{align*}
and so
\begin{align*}
|\rho_1\cos(\eta_1)-\rho_2 \cos(\eta_2)|\lesssim &\varepsilon\left(|\alpha_1-\alpha_2|  + \|\widetilde{\lambda}_1-\widetilde{\lambda}_2\|_{H^2_{\delta}}
+\|\widetilde{H}_1-\widetilde{H}_2\|_{\q H^1_{\delta+1}}\right), \\
|\rho_1\sin(\eta_1)-\rho_2 \sin(\eta_2)|\lesssim  &\varepsilon\left(|\alpha_1-\alpha_2| + \|\widetilde{\lambda}_1-\widetilde{\lambda}_2\|_{H^2_{\delta}}
+\|\widetilde{H}_1-\widetilde{H}_2\|_{\q H^1_{\delta+1}}\right) .
\end{align*}
We decompose
$$H'_i = H^{(1)}_i +H^{(2)}_i + H^{(3)}_i,$$
where $ H^{(1)}_i, H^{(2)}_i, H^{(3)}_i$ satisfy equations \eqref{h1}, \eqref{h2} and \eqref{h3}. We have seen during the proof of Proposition \ref{prph2} that we can write
$$ H^{(2)}_i= b(\theta)H_\theta+\wht H^{(2)}_i,$$
where $\wht H^{(2)}_i$ satisfies
$$\partial_l\left( \wht H^{(2)}_i\right)_{lj}= \frac{b(\theta)\chi'(r)}{r}\left(
\begin{array}{l}\cos(\theta)\\ \sin(\theta)\end{array}\right).$$
We have seen during the proof of Proposition \ref{prph3} that we can write
$$H^{(3)}_i=-\frac{\rho_i \chi(r)}{4r}\left(\begin{array}{cc}
                                          \cos(3\theta-\eta_i) & \sin(3\theta-\eta_i) \\ 
                                           \sin(3\theta-\eta_i) & -\cos(3\theta-\eta_i)
                                         \end{array} \right)+\widetilde{H}^{(3)}_i,$$
where    $\widetilde{H}^{(3)}_i$ satisfies
$$\partial_l\left(\widetilde{H}^{(3)}_i\right)_{lj}=\frac{\rho_i\chi'(r)}{2r}   \left(
\begin{array}{l}\cos(2\theta-\eta_i)\\ \sin(2\theta-\eta_i)\end{array}\right).$$                                 
Therefore we can write 
$$H'_i=b(\theta)H_\theta - \frac{\rho_i \chi(r)}{4r}\left(\begin{array}{cc}
                                          \cos(3\theta-\eta_i) & \sin(3\theta-\eta_i) \\ 
                                           \sin(3\theta-\eta_i) & -\cos(3\theta-\eta_i)
                                         \end{array} \right) + K_i,$$
where $K_i$ satisfies
\begin{align*}
\partial_l (K_i )_{lj} =& -\dot{u}\partial_j u +\frac{1}{2}\partial_j \widetilde{\tau} -\frac{1}{2}\widetilde{\tau}\partial_j \lambda_i
-(\widetilde{H}_i)_{lj}\partial_l \lambda_i \\
 & +(g_i)_{j}  -\partial_l \widetilde{\lambda_i}((b(\theta)+\rho_i\cos(\theta-\eta_i))H_\theta)_{lj} -\frac{1}{2}\chi(r)\frac{b+\rho_i\cos(\theta-\eta_i)}{r}\partial_j \widetilde{\lambda_i},\\
\end{align*}
with 
$$\left(\begin{array}{l}
   (g_i)_1\\ (g_i)_2
  \end{array}\right)
=\frac{\chi'(r)}{r}\left(\begin{array}{l}
   b(\theta)\cos(\theta) +\frac{\rho_i}{2}\cos(2\theta-\eta_i)+\frac{\rho_i}{4}\cos(\eta_i)\\ 
b(\theta)\sin(\theta)+\frac{\rho_i}{2}\sin(2\theta-\eta_i)+\frac{\rho_i}{4}\cos(\eta_i)
  \end{array}\right).$$
  Consequently $K_1-K_2$ satisfies
  \begin{align*}
  \partial_l (K_1-K_2)_{lj}=&-\frac{1}{2}\wht \tau\partial_j(\lambda_1-\lambda_2)-(\wht H_1)_{lj}\partial_j \lambda_1
  +(\wht H_2)_{lj}\partial_j \lambda_2+(g_1)_j-(g_2)_j\\
  &-\partial_l\wht \lambda_1((b(\theta)+\rho_1\cos(\theta-\eta_1))H_\theta)_{lj}
  +\partial_l\wht \lambda_2((b(\theta)+\rho_2\cos(\theta-\eta_2))H_\theta)_{lj}\\
  &-\frac{1}{2}\chi(r)\frac{b+\rho_1\cos(\theta-\eta_1)}{r}\partial_j \widetilde{\lambda_1}+\frac{1}{2}\chi(r)\frac{b+\rho_2\cos(\theta-\eta_2)}{r}\partial_j \widetilde{\lambda_2}.
  \end{align*}
The right-hand side is in $H^0_{\delta+2}$
and satisfies
\begin{align*}
\|\partial_l (K_1-K_2)_{lj}\|_{H^0_{\delta+2}} 
\lesssim &\varepsilon \left(|\alpha_1-\alpha_2| + \|\widetilde{\lambda}_1-\widetilde{\lambda}_2\|_{H^2_{\delta}}
+\|\widetilde{H}_1-\widetilde{H}_2\|_{\q H^1_{\delta+1}}\right) \\
&+|\rho_1\cos(\eta_1)-\rho_2 \cos(\eta_2)| +|\rho_1\sin(\eta_1)-\rho_2 \sin(\eta_2)|\\
\lesssim &\varepsilon\left(|\alpha_1-\alpha_2| + \|\widetilde{\lambda}_1-\widetilde{\lambda}_2\|_{H^2_{\delta}}
+\|\widetilde{H}_1-\widetilde{H}_2\|_{\q H^1_{\delta+1}}\right) .
\end{align*}
Therefore, we can apply Lemma \ref{lemme} which yields
$$K_1-K_2=   \frac{m \chi(r)}{r}\left(\begin{array}{cc}
                                          \cos(\theta+\phi) & \sin(\theta+\phi) \\ 
                                           \sin(\theta+\phi) & -\cos(\theta+\phi)
                                         \end{array} \right) + \widetilde{K},$$
with $\widetilde{K}\in \q H^1_{\delta+1}$ which satisfies
$$\|\widetilde{K}\|_{\q H^1_{\delta+1}}\lesssim \varepsilon\left(|\alpha_1-\alpha_2| + \|\widetilde{\lambda}_1-\widetilde{\lambda}_2\|_{H^2_{\delta}}
+\|\widetilde{H}_1-\widetilde{H}_2\|_{\q H^1_{\delta+1}}\right) .$$
Moreover, since we can also write
$$H'_i=(b(\theta)+\rho_i \cos(\theta-\eta_i))H_\theta+\wht H'_i,$$
we have
\begin{align*}
K_1-K_2=&-\frac{\chi(r)}{4r}\left(\begin{array}{cc}
                                          \rho_1\cos(\theta+\eta_1)-\rho_2\cos(\theta+\eta_2) & \rho_1\sin(\theta+\eta_1)-\rho_2\sin(\theta+\eta_2) \\ 
                                       \rho_1\sin(\theta+\eta_1)-\rho_2\sin(\theta+\eta_2)  & -\rho_1\cos(\theta+\eta_1)+\rho_2\cos(\theta+\eta_2) 
                                         \end{array} \right)\\
                                         &+\wht H'_i-\wht H'_2\\
= &\frac{M \chi(r)}{r}\left(\begin{array}{cc}
                                          \cos(\theta+\Phi) & \sin(\theta+\Phi) \\ 
                                           \sin(\theta+\Phi) & -\cos(\theta+\Phi) \end{array}\right) +   \wht H'_i-\wht H'_2, 
                                           \end{align*}
where $M(\cos(\Phi),\sin(\Phi))=-\frac{1}{4}(\rho_1\cos(\eta_1)-\rho_2\cos(\eta_2),
\rho_1\sin(\eta_1)-\rho_2\sin(\eta_2))$.                                      By uniqueness of the decomposition given by Lemma \ref{lemme}, we obtain
$$\widetilde{K}= \widetilde{H}'_1-\widetilde{H}'_2,$$ 
and we deduce
\begin{equation}\label{difh}
\|\widetilde{H}'_1-\widetilde{H}'_2\|_{\q H^1_{\delta+1}}\lesssim \varepsilon\left(|\alpha_1-\alpha_2| + \|\widetilde{\lambda}_1-\widetilde{\lambda}_2\|_{H^2_{\delta}}
+\|\widetilde{H}_1-\widetilde{H}_2\|_{\q H^1_{\delta+1}}\right) .
\end{equation}
Last we estimate
\begin{align*}
\|\Delta \lambda'_1-\Delta \lambda'_2\|_{H^2_\delta} \lesssim&
\varepsilon \|\widetilde{H}_1-\widetilde{H}_2\|_{\q H^1_{\delta+1}}
 + \varepsilon \left(|\rho_1\cos(\eta_1)-\rho_2 \cos(\eta_2)| +|\rho_1\sin(\eta_1)-\rho_2 \sin(\eta_2)|\right)\\
\lesssim &\varepsilon\left(|\alpha_1-\alpha_2| + \|\widetilde{\lambda}_1-\widetilde{\lambda}_2\|_{H^2_{\delta}}
+\|\widetilde{H}_1-\widetilde{H}_2\|_{\q H^1_{\delta+1}}\right). 
\end{align*}
Thus, thanks to Corollary \ref{coro}, we have
\begin{equation}\label{difl}
|\alpha'_1-\alpha'_2| + \|\widetilde{\lambda'}_1- \widetilde{\lambda'}_2\|_{H^2_\delta}
\lesssim \varepsilon\left(|\alpha_1-\alpha_2| + \|\widetilde{\lambda}_1-\widetilde{\lambda}_2\|_{H^2_{\delta}}
+\|\widetilde{H}_1-\widetilde{H}_2\|_{\q H^1_{\delta+1}}\right).
\end{equation}\\

In view of \eqref{difh} and \eqref{difl}, we have proved that for $\varepsilon$ small enough there exists
$\wht C$  such that for $(\alpha_i, \wht \lambda_i, \wht H_i) \in B_X(0,2C\varepsilon)$ we have
$$\|F((\alpha_1, \wht \lambda_1, \wht H_1))-F((\alpha_2, \wht \lambda_2, \wht H_2))\|_X\leq \wht C\varepsilon\|(\alpha_1, \wht \lambda_1, \wht H_1)-(\alpha_2, \wht \lambda_2, \wht H_2)\|_X.$$
In particular, if we apply this to $(\alpha_2, \wht \lambda_2, \wht H_2)=(0,0,0)$ we obtain
$$\|F((\alpha_1, \wht \lambda_1, \wht H_1))-C\varepsilon\|_X\leq 2\wht CC\varepsilon^2.$$ 
Therefore if $\varepsilon$ is such that
$2\wht C\varepsilon \leq 1$, $F$ is a contraction map from $B_X(0,2C\varepsilon)$ to itself.
Therefore, $F$ has a unique fixed point $(\alpha, \widetilde{\lambda}, \widetilde{H})$.
The estimates of Lemma \ref{rho} and Proposition \ref{lichn} complete the proof of Theorem \ref{main}.

\paragraph{Acknowledgement}
The author would like to thank Rafe Mazzeo for stimulating discussions concerning this problem, and also Jérémie Szeftel for his attentive reading of this 
article and his kind supervising.

\bibliographystyle{plain}
\bibliography{contrainte}

\end{document}